\PassOptionsToPackage{numbers}{natbib}
\documentclass[graybox,natbib,envcountsame]{svmult}

\usepackage{mathptmx}
\usepackage{helvet}
\usepackage{courier}
\usepackage{makeidx}
\usepackage{graphicx}
\usepackage{multicol}
\usepackage[bottom]{footmisc}

\usepackage{amsmath}   
\usepackage{amssymb} 
\usepackage[utf8]{inputenc}
\usepackage{tikz}
\usetikzlibrary{%
  cd, 
  matrix 
}

\usepackage{amsfonts,latexsym,verbatim,amsbsy}
\usepackage{dsfont,bm}
\usepackage{graphicx}
\usepackage{grffile}
\usepackage{bbm}
\usepackage{enumerate}   
\usepackage{array}
\usepackage{pifont}
\usepackage{tabularx}
\usepackage{enumitem}

\spnewtheorem*{AAT}{Abel's Addition Theorem}{\bfseries}{\itshape}

\smartqed
\usepackage{etoolbox}
\patchcmd{\qed}{\ifmmode\qedsymbol}{\ifmmode\the\qedsymbol}{}{\foobar}
\patchcmd{\qed}{\hfil\qedsymbol}{\hfil\the\qedsymbol}{}{\foobar}
\qedsymbol{\proofSymbol}

\newcommand{\qedhere}{\tag*{\the\qedsymbol}}

\definecolor{labelkey}{rgb}{0,0,1}

\newcommand{\sect}[1]{Section~\ref{#1}}

\newcommand{\R}{\ensuremath{\mathbb{R}}}   
\newcommand{\T}{\ensuremath{\mathbb{T}}}   

\def \d {\delta}

\def \n {\nabla}

\def \O {\Omega}

\def \bd {{\bf d}}

\def \bu {{\bf u}}

\def \cA {\mathcal{A}}

\def \cE {\mathcal{E}}

\def \cH {\mathcal{H}}

\def \cL {\mathcal{L}}

\newcommand{\ave}[1]{ \left[ #1 \right]}

\DeclareMathOperator{\supp}{supp} %

\DeclareMathOperator{\diam}{diam} %
\def \lan {\langle}
\def \ran {\rangle}

\renewcommand{\geq}{\geqslant}

\renewcommand{\leq}{\leqslant}

\newcommand{\cmark}{\ding{51}}%
\newcommand{\xmark}{\ding{55}}%

\def \dx  {\, \mbox{d}x}

\def \dt  {\, \mbox{d}t}

\def \dy  {\, \mbox{d}y}

\def \dr  {\, \mbox{d}r}
\def \ds  {\, \mbox{d}s}

\def \ddt  {\frac{\mbox{d\,\,}}{\mbox{d}t}}

\def\R{\mathbb{R}}

\def\T{\mathbb{T}}

\def\Hmdot{\dot{H}^m}
\def\Hldot{\dot{H}^l}
\def\Hkdot{\dot{H}^k}

\def\u{\textbf{u}}
\def\uavg{\u_F}

\def\dl{\partial^l}
\def\bd{\textbf{d}}
\def\V{\textbf{V}}

\def \st {\mathrm{s}}
\def \wt {\mathrm{w}}

\begin{document}

\title*{Well-posedness and long time behavior of the Euler Alignment System with adaptive communication strength}

%
%
\author{Roman Shvydkoy \and Trevor Teolis}
\institute{
  R.~Shvydkoy \at Department of Mathematics, Statistics and Computer Science, University of
  Illinois at Chicago, 851 S Morgan St, M/C 249, Chicago, IL 60607
  \email{shvydkoy@uic.edu} \and %
  T.~Teolis \at Department of Mathematics, Statistics and Computer Science, University of
  Illinois at Chicago, 851 S Morgan St, M/C 249, Chicago, IL 60607, 
  \email{tteoli2@uic.edu}
}

\maketitle

\thanks{\textbf{Acknowledgment.}  
	This work was  supported in part by NSF
	grant  DMS-2107956.}

\begin{abstract} {
We study a new flocking model which has the versatility to capture the physically realistic qualitative behavior of the Motsch-Tadmor model, while 
also retaining the entropy law, which lends to a similar 1D global well-posedness analysis to the Cucker-Smale model.  
This is an improvement to the situation in the Cucker-Smale case, which may display the physically unrealistic behavior that large flocks overpower the dynamics of small, far away flocks; 
and it is an improvement in the situation in the Motsch-Tadmor case, where 1D global well-posedness is not known.
The new model was proposed in \cite{shvydkoy2022environmental} and has a similar structure to the Cucker-Smale and Motsch-Tadmor hydrodynamic systems, but 
with a new feature: the communication strength is not fixed, but evolves in time according 
to its own transport equation along the Favre-filtered velocity field.  This transport of the communication strength is precisely what preserves the entropy law. 
A variety of phenomenological behavior can be obtained from various choices of the initial communication strength, including the aforementioned Motsch-Tadmor-like behavior. 
We develop the general well-posedness theory for the new model and study the long time behavior-- including alignment, strong flocking in 1D, and entropy estimates to 
estimate the distribution of the limiting flock, all of which extend the classical results of the Cucker-Smale case. 
In addition, we provide numerical evidence to show the similar qualitative behavior
between the new model and the Motsch-Tadmor model for a particular choice of the initial communication strength.   
}

 
\subclassname {37A60, 92D50}

\date{\today}

\end{abstract}

\tableofcontents

\section{Introduction}
\subsection{Brief background and motivation}
The pressureless Euler Alignment system based on the classical Cucker-Smale model is given by 
\begin{align}
    \label{EAS_CS}
    \tag{CS}
    \begin{cases}
        \partial_t \rho + \nabla \cdot (\u \rho) = 0 \\
        \partial_t \u + \u \cdot \nabla \u = (\u\rho)_{\phi} - \u\rho_{\phi}
    \end{cases}
\end{align}
where we use the shorthand notation $f_{\phi} := f \ast \phi$ for convolutions, see \cite{CS2007a,HT2008}. Here, $\phi\in C^1$ is a smooth non-negative radially decreasing communication kernel, 
$\rho, \u$ are density and velocity of the flock, respectively, and the environment in question is either $\O = \T^n$ or $\R^n$ 
(although our results for the $\st$-model will be stated only for the Torus $\T^n$).

Analysis and relevance to applications of \eqref{EAS_CS} has been the subject of many studies in recent years, see \cite{ABFHKPPS,FK2019,HeT2017,Sbook,TT2014} and references therein. In particular, flocking in the classical sense of uniformly bounded radius and exponential alignment
\begin{equation}\label{e:flocking}
	\sup_{t\geq 0} (\diam(\supp \rho )) < \infty , \quad	\sup_{x \in\supp \rho} |\u(t,x) - \u_\infty|\leq C_0 e^{- \d t}
\end{equation}
holds under ``heavy-tail" condition on the kernel, \cite{TT2014}
\begin{equation}\label{heavy_tail}
\int_0^\infty \phi(r) \dr = \infty,
\end{equation}
by direct analogy with the agent-based result of Cucker and Smale \cite{CS2007a,CS2007b,HL2009}. Here, the limiting velocity $\u_\infty$ is determined by the initial momentum, which is conserved.

The alignment force in the system is mildly diffusive as seen for instance from the energy balance law
\begin{equation}\label{e:enlaw}
\ddt \frac{1}{2} \int_\O \rho |\u|^2 \dx = -  \int_{\O \times \O} \phi(x-y) |\u(x) - \u(y)|^2 \rho(x)\rho(y) \dy \dx.
\end{equation}
Therefore the regularity theory for \eqref{EAS_CS} in the smooth communication case runs somewhat parallel to hyperbolic conservation laws; the difference being that 
there is room for regularization effect in the force. For instance, in 1D, Carrillo, Choi, Tadmor and Tan \cite{CCTT2016} establish an exact threshold regularity criterion in terms of the so called ``e-quantity"
\begin{equation}\label{e:e}
e = \partial_x u + \rho_\phi, \qquad \partial_t e + \partial_x (u e) = 0.
\end{equation}
The solution with smooth initial condition remains smooth if and only if $e_0 \geq 0$. In multi-D, partial results are found in \cite{HKK2014,HKK2015,HeT2017,TT2014} and the book \cite{Sbook} presents a general continuation criterion in the spirit of Grassin \cite{Grassin99}, Poupaud \cite{Poupaud99}: as long as 
	\begin{equation}\label{e:BKMdiv}
\int_0^{T_0} \inf_{x\in \O} \n \cdot \bu(t,x) \dt > - \infty
\end{equation}
the solution can be continued smoothly beyond $T_0$. 

Phenomenologically the Euler alignment system performs well when the flock is mono-scale. 
For instance, in the Darwin mission, a constellation of satellites are coordinated to remain equidistant from one another (i.e. mono-scale).  
It is shown in \cite{Darwin} that the Cucker-Smale dynamics can used as a control law for the satellites in order to maintain this formation. 
However, in heterogeneous formations, when two remote clusters of largely diverse size appear, the dynamics according to \eqref{EAS_CS} yields pathological results. 
The large cluster hijacks evolution of the smaller cluster removing any fine features of the latter. 
Motsch and Tadmor argue in \cite{MT2011,MT2014} that rebalancing the averaging operation in the alignment force cures such issues. 
They proposed the following modification
\begin{align}
    \label{EAS_MT}
    \tag{MT}
    \begin{cases}
        \partial_t \rho + \nabla \cdot (\u \rho) = 0 \\
        \partial_t \u + \u \cdot \nabla \u = \frac{1}{\rho_{\phi}} \big( (\u\rho)_{\phi} - \u\rho_{\phi} \big).
    \end{cases}
\end{align}
The model has the exact same flocking behavior \eqref{e:flocking} under heavy-tail kernel, but progress in well-posedness theory of the system has been stalled even in 1D due to the lack of the energy law \eqref{e:enlaw} or the e-quantity \eqref{e:e}. 
Therefore the need for a model with qualitative features similar to those of Motsch-Tadmor but better analytical properties has become a pressing problem.

A model with the potential to achieve these features has been proposed in \cite{shvydkoy2022environmental} in the context of Environmental Averaging models, but it 
has not received any scrutiny there.  The goal of this paper is to show that the proposed model, which we call the adaptive strength model, or $\st$-model for short, does indeed possess the desirable qualitative and analytic properties.  We describe the $\st$-model next. 

\subsection{Environmental Averaging Models and the $\st$-model}

Despite their differences both systems \eqref{EAS_CS} and  \eqref{EAS_MT} share similar structure of the alignment force. It can be written as 
\begin{equation}\label{e:F}
F = \st_\rho ( \ave{\u}_\rho - \u),
\end{equation}
where $\ave{\u}_\rho$ is the averaging component, and $\st_\rho \geq 0$ is the communication strength. In both cases $\ave{\u}_\rho= \frac{(\u \rho)_\phi}{\rho_\phi}$, which  is also known in turbulence literature as the Favre filtration, see \cite{Favre}. 
The difference comes only in the prescription of the communication strength. 
In the Cucker-Smale case $\st_\rho = \rho_\phi$, while in the Motsch-Tadmor case $\st_\rho \equiv 1$. Many other examples encountered in the literature, including multi-flocks and multi-species, share the same structure and fall under the category of so called `environmental averaging models'. The general theory of such models has been developed in \cite{shvydkoy2022environmental}. 
The alignment characteristics and well-posedness are determined by 
a strength function $\st_{\rho}$ and the weighted averaging operator $\st_\rho [\cdot]_{\rho}$. 

It is observed  in  \cite{shvydkoy2022environmental} that the main reason why the e-equation \eqref{e:e} holds in the Cucker-Smale case is because, for this model, the strength function $\rho_{\phi}$ happens to evolve according to its own transport equation along the Favre-averaged field:
\begin{equation}\label{ }
\partial_t \rho_\phi + \partial_x ( \rho_\phi \ave{\u}_\rho ) = 0.
\end{equation}
So, a new model was proposed where instead of prescribing communication strength $\st_\rho$ a priori, one lets it adapt to the environment through transport along the averaged field
\begin{equation}\label{ }
\partial_t \st + \partial_x ( \st \ave{\u}_\rho) = 0, \hspace{5mm} \st \geq 0 
\end{equation}
As such, the adaptive strength $\st$ becomes another unknown, and it may not explicitly depend on the density. The resulting full model, which we call $\st$-model for short, reads
\begin{equation}
    \label{s_model}
    \tag{SM}
    \begin{cases}
        \partial_t \rho + \nabla \cdot (\u\rho) = 0 \\
        \partial_t \st + \nabla \cdot (\st \ave{\u}_{\rho}) = 0,  \hspace{15mm} \st \geq 0   \\
        \partial_t \u + \u \cdot \nabla \u = \st (\ave{\u}_{\rho} - \u),
    \end{cases}
\end{equation}
Now, regardless of the particular averaging used, the model always admits a conserved 
quantity in 1D similar to \eqref{e:e}, which lands it more amenable well-posedness analysis.   
Indeed, if we differentiate the velocity equation in $x$, we get 
\begin{equation*}
    \partial_t \partial_x u + \partial_x u (\partial_x u + \st) + u(\partial_x^2 u + \partial_x \st) = \partial_x (\st [u]_{\rho}) .
\end{equation*} 
If $\st$ satisfies the transport equation in \eqref{s_model}, then it becomes 
\begin{equation*}
    \partial_t (\partial_x u + \st) + \partial_x u (\partial_x u + \st) + u(\partial_x^2 u + \partial_x \st) = 0
\end{equation*} 
which is the desired conservation law (a.k.a. the entropy law):  
\begin{equation}
    \label{conservation_law}
    e = \partial_x u + \st, \hspace{5mm} \partial_t e + \partial_x (ue) = 0.
\end{equation}
In the Cucker-Smale theory, this extra conservation law holds the key to 1D global well-posedeness, strong flocking, and 
entropy estimates on the limiting distribution of the flock.  We affirm in this paper that these results can be extended to the case 
of the $\st$-model (albeit with the specific case of the Favre averaging, which is the most relevant to the Cucker-Smale and Motsch-Tadmor models;
see Section \ref{intro:main_results} for further justification for working with the case of the Favre averaging). 

\begin{remark}
    The $e$-quantity is sometimes referred to as an entropy because, in 1D, it's magnitude provides a measure of distance 
    of the limiting flock from the uniform distribution.  For the Cucker-Smale model, this was proved in \cite{LS-entropy} 
    and is extended to the $\st$-model with Favre averaging here in Theorem \ref{thm:entropy_estimate_intro}. 
\end{remark}

We note that any attempt to develop well-posedness and alignment analysis of \eqref{s_model} for most general averaging operators $[\u]_{\rho}$ necessitates many technical assumptions and therefore leads to an obscure exposition. 
So, to avoid such technicalities and to keep our analysis close to the Cucker-Smale and Motsch-Tadmor cases, we limit ourselves to the Favre-based models, setting $\u_F := (\u \rho)_{\phi} / \rho_{\phi}$ to be our fixed averaging protocol.
With regard to the local and 1D global well-posedness results, the choice of $\u_F$ is merely convenient and the results can be extended to general averaging protocols $[\u]_{\rho}$. 
Notably, however, the small data and long-time behavior results depend on the explicit structure of $\u_F$ and therefore these results may not be extendable to general averaging protocols. 
Fortunately, choosing the specific averaging $\u_F$ over general averaging operators is a small sacrifice.
Indeed, even with the $\st$-model with the specific Favre averaging has versatility to capture the Motsch-Tadmor-like behavior while retaining the nice analytic properties of the 
Cucker-Smale model (owing to the conservation law \eqref{conservation_law}).
We will from here on refer to the $\st$-model with Favre averaging as just the '$\st$-model', unless it is stated otherwise. 

To rewrite the $\st$-model in a simpler form and more explicit form,  
we introduce the new variable 
\[
\wt := \frac{\st}{\rho_{\phi}}.
\]
  As $\rho_{\phi}$ and $\st$ satisfy the 
same continuity equation, $\wt$ satisfies the pure transport equation along the characteristics of $\u_F$
\begin{equation*}
    \partial_t \wt + \u_F \cdot \nabla \wt = 0.
\end{equation*}
 We will  refer to it as the ``weight" in order to distinguish it from the strength in the $\st$-model.  The $\st$-model can now be written in a way that eliminates division by $\rho_\phi$ in the alignment force:
\begin{align*}
    \label{EAS_WM}
    \tag{WM}
    \begin{cases}
        \partial_t \rho + \nabla \cdot (\u\rho) = 0 \\
        \partial_t \wt + \u_F \cdot \nabla \wt = 0 \\
        \partial_t \u + \u \cdot \nabla \u  = \wt ((\u \rho)_{\phi} - \u \rho_{\phi}).
    \end{cases}
\end{align*}

Setting $\wt =1$ we recover the Cucker-Smale case \eqref{EAS_CS}, while setting $\wt_0 = 1/({\rho_0})_{\phi}$, at least initially we recover the Motsch-Tadmor data. In the latter case, as the strength evolves, it will deviate from the Motsch-Tadmor strength. 
The question arises as to whether this new strength still retains the same balancing properties as the original Motsch-Tadmor model. In Section \ref{description_of_numerics} we present numerical evidence that it is indeed the case -- 
small flocks are not overly influenced by large far away flocks. 


\subsection{Notation}
Before stating the main results, we will describe notational conventions used throughout the paper.
We will use $\partial^k_x$ and $\partial^k_t$ to denote the $k^{th}$ partial derivative with respect to time and space, respectively. 
In multi-D, we will at times use $\partial$ to denote an arbitrary spatial derivative (i.e. the partial derivative in an arbitrary 
coordinate direction).  Since the kernel $\phi$ is a radial function, we will denote the derivatives by $\phi'$, $\phi''$, etc. 
We will use $c$'s to denote lower bounds.  For instance, $c_0$, $c_1$, and $c_2$, will refer to lower bounds on $\rho$, $\phi$, 
and $e$, respectively. To abbreviate maximum and minimum values of a function, we write $f_+ := \sup_{x \in \T^n} f(x)$ and $f_- = \inf_{x \in \T^n} f(x)$.
As mentioned in the introduction, we abbreviate convolutions by $f_{\phi} := f \ast \phi$.
Regarding function spaces, $H^m := H^m(\T^n)$ is the Sobolev space of functions whose first $m$ derivatives (defined in the weak sense) lie in $L^2(\T^n)$.  
We will denote by $L^1_+$ the space of non-negative $L^1(\T^n)$ functions.  Finally, $C_w([0,T]; X)$ denotes the space of weakly continuous functions with values in $X$ on the time interval $[0,T]$.

\subsection{The scope and main results}
\label{intro:main_results}
Let us now state the main results. 
It will be more convenient to develop regularity theory for the $\st$-model in the form \eqref{EAS_WM}, treating $\wt$ as an unknown.   
In section \ref{LWP}, local existence and continuation is proved in higher regularity Sobolev classes via a Banach Fixed Point argument 
for a viscous regularization of \eqref{EAS_WM}; the full result is then obtained by compactness arguments.  
Energy estimates are also established that give rise to to the continuation criterion.  
The assumptions (A1)-(A5) required for the local existence are stated below in Theorem \ref{lwp} and will be used throughout the paper. 
We will indicate, if possible, how our results can be extended to the open space $\R^n$.
\begin{theorem} [Local well-posedness]  \label{lwp}
    Suppose the following assumptions hold. 
    \begin{enumerate} [leftmargin=0.8cm, label=(A\arabic*)]
        \item The domain is the torus, $\Omega = \T^n$ 
        \item The kernel $\phi$ is a smooth, non-negative, radial, and decreasing function of the distance 
        \item The density and weight are non-negative functions (i.e. $\rho, \wt \geq 0$)
        \item The initial data $(\rho_0, \wt_0, \u_0) \in (H^k \cap L_+^1) \times H^l \times H^m$ with $l \geq m \geq k+1 \geq n/2 + 3$
        \item There is a constant $c_0$ such that $(\rho_0)_{\phi} \geq c_0 > 0$
    \end{enumerate}
    Then there exists a time $T > 0$ and a unique solution $(\rho, \wt, \u) \in C_w([0,T]; (H^k \cap L_+^1) \times H^l \times H^m)$
    to \eqref{EAS_WM} satisfying the initial condition and $\inf_{[0,T]} \rho_{\phi} >0$.
    
Moreover, if 
\begin{equation}
    \label{cont_criterion}
    \int_0^{T}\left( \|\nabla \u\|_{\infty} + \Big\|\frac{1}{\rho_{\phi}} \Big\|_{\infty}^{l+1} \right) \ds < \infty
\end{equation}
then the solution can be continued beyond time $T$.
\end{theorem}

\begin{remark}
    Provided, $\phi \geq c_1 > 0$, Theorem \ref{lwp} also holds in $\R^n$. 
    The proof relies on a lower bound on $\rho_{\phi}$, which necessitates 
    a lower bound on $\phi$ when $\rho \in L^1(\R^n)$.
\end{remark}

The continuation criterion can be used to obtain a small data result.  
See Theorem \ref{small_data} for the complete statement. 
\begin{theorem} [Global well-posedness for small initial data in multi-D]
    \label{thm:small_data_intro}
    Assume (A1)-(A4). 
    If in addition, the kernel $\phi$ is bounded below, $\phi \geq c_1 > 0$ (which implies (A5)), and the initial velocity and initial variation of the velocity are small enough, 
    then there is a unique solution $(\rho, \wt, \u) \in C_w([0,T]; (H^k \cap L^1_+) \times H^l \times H^m)$ to \eqref{EAS_WM} existing globally in time.
\end{theorem}

In 1D, having the additional conservation law \eqref{conservation_law} helps to establish control over $\partial_x u$ first, and then over decay rate of $\rho_\phi$ in order to achieve the same threshold criterion for global well-posedness as in the classical Cucker-Smale case. 
In fact, we extend this result to multi-dimensional unidirectional flocks introduced in \cite{LS-uni1}
\begin{equation}\label{e:uniintro} 
    \u(x,t) = u(x,t) \bd, \hspace{8mm} \bd \in \mathbb{S}^{n-1}, \hspace{2mm} \u: \T^n \times \R^+ \to \R.
\end{equation}
The key feature of these solutions is possession of the same conservation law \eqref{conservation_law} 
\[
e = \n u \cdot \bd + \st
\]
although in this case it does not give control over the full gradient of $\bu$. In \sect{UniGWP}, we present a bootstrap argument that establishes full control provided the weight $\wt$ is bounded above and below. 
See Theorem \ref{t:GWP_1D} and Theorem \ref{t:uni} for the complete statement of the 1D and multi-D cases, respectively.
\begin{theorem} [Global well-posedness for unidirectional flocks]
\label{thm:UniGWP}
    Assume (A1)-(A5) and that the initial density is non-vacuous, i.e. $\rho_0 \geq c > 0$. If in addition, $\u_0$ is unidirectional (i.e. of the form \eqref{e:uniintro}), 
    then there's a unique global solution provided $e_0 \geq 0$.
\end{theorem}

Turning to long time behavior, we note that there is exponential $L^{\infty}$-based alignment when the kernel is bounded below, see Theorem \ref{thm:linf_alignment}. 
The proof is analogous to the Cucker-Smale case given by Ha and Liu in \cite{HL2009} so we don't include it as a main result; but it is an important one as it shows 
that the new model retains strong alignment characteristics. Additionally, $L^{\infty}$-based alignment will be used for the 
small data and strong flocking results. In Section \ref{local_alignment}, for local communication kernels, we show conditional alignment of the velocity in the $L^2$ sense
(as opposed to the unconditional $L^2$-based alignment result in the Cucker-Smale case).  
Let $V_2(t) = \frac{1}{2} \int |\u(t,x) - \bar{\u}(t)|^2 \rho(t,x) \dx$, where $\bar{\u}(t) = \frac{1}{M} \int_{\T^n} \u(t,x) \rho(t,x) \dx$ is the average momentum at time $t$ 
and $M = \int_{\T^n} \rho_0(x) \dx$ is the mass of the flock.
\begin{theorem} [Alignment in $L^2$ under local communication]
    \label{thm:l2_alignment_intro}
    Assume (A1)-(A5).
    For smooth initial data and local kernels $\phi(x,y) \geq c_1 \mathbbm{1}_{|x-y|<r_0}$, there exists $c_1' := c_1'(r_0, c_1) > 0$ such that 
    if the solution satisfies 
    \begin{equation}
        \label{ineq:dissipation_condition}
        0 < \frac{\wt_+ - \wt_-}{\wt_-} \leq \frac{c_1'}{M \|\phi\|_{\infty}} \frac{\rho_-^2(t)}{\rho_+(t)}, \hspace{5mm} t \geq 0 
    \end{equation}
    Then there is exponential $V_2$-based alignment. In other words, there exists a constant $\delta > 0$ such that  
    \begin{equation*}
        V_2(t) \leq V_2(0) e^{-\delta t}
    \end{equation*}
\end{theorem}
\begin{remark}
    Non-integrability of $\frac{\rho_-^2}{\rho_+}$ is the key for alignment under local kernels. This was first observed by 
    Tadmor in \cite{Tadmor-notices} in the case of the Cucker-Smale model.  
    However, when $\wt$ is non-constant, the constraint is a (more stringent) uniform lower bound on $\frac{\rho_-^2}{\rho_+}$ and thus non-integrability is automatic.
    Notably, $V_2$-based alignment in the Cucker-Smale case with the relaxed non-integrability assumption is not necessarily exponential.
\end{remark}
With alignment and a threshold condition for 1D global well-posedness in hand, the question arises as to whether the density converges to a limiting distribution (a.k.a strong flocking). 
In Section \ref{sec:strong_flocking}, we affirm this is the case in 1D, provided the entropy $e_0$ and the kernel $\phi$ both bounded away from zero.
\begin{theorem} [Strong Flocking in 1D]
    \label{thm:strong_flocking_intro}
    Assume (A1)-(A4). If in addition, the kernel and $e_0$ are bounded below, i.e. $\phi \geq c_1 > 0$ (which implies (A5)) and $e_0 = \partial_x u_0 + \wt_0 (\rho_0)_{\phi} \geq c_2 > 0$, 
    for some constants $c_1$, $c_2$, then 
    there exists a global in time solution with a limiting velocity $u_{\infty}$. In particular, there exists $\delta > 0$ such that 
    \begin{equation*}
        \|u - u_{\infty} \|_{\infty} + \|\partial_x u  \|_{\infty} + \|\partial_x^2 u \|_{\infty} \leq e^{-\delta t} 
    \end{equation*}
    As a consequence, there exists a limiting density distribution $\rho_{\infty}$ such that 
    \begin{equation*}
        \|\rho(t, \cdot) - \rho_{\infty}(\cdot - t u_{\infty} ) \|_{\infty} \leq e^{-\delta t}
    \end{equation*}
\end{theorem}

It is not known, even in the Cucker-Smale case, what the limiting distribution $\rho_{\infty}$ looks like. 
However, the $L^1$ distance from the uniform distribution $\bar{\rho} = \frac{M}{2\pi}$ can be estimated using relative entropy estimates. 
In Section \ref{sec:entropy}, we establish the following theorem, which is an extension of the result by Leslie and Shvydkoy for the 
Cucker-Smale case established in \cite{LS-entropy}.
\begin{theorem} [Deviation of limiting flock from the uniform distribution]
    \label{thm:entropy_estimate_intro}
    Assume (A1)-(A5) and that $\phi$ satisfies $\phi(r) \geq \mathds{1}_{r \leq r_0}$, and $e_0 = \partial_x(u_0) + \wt_0 (\rho_0)_{\phi} \geq 0$.
    Let $\tilde{e} = \partial_x u + \cL_{\phi} \rho$, where $\cL_{\phi} \rho = \wt(x) \int_{\T} (\rho(y) - \rho(x)) \phi(x-y) \dy$ 
    and $\tilde{q} = \frac{\tilde{e}}{\rho}$. 
    If $\|\tilde{q}\|_{\infty} \leq \tilde{Q} < \wt_+ \|\phi\|_{L^1}$ for some constant $\tilde{Q}$, 
    then there exists a constant $c := c(r_0)$ such that 
    \begin{equation*}
        \limsup_{t \to \infty} \|\rho(t) - \bar{\rho} \|_{L^1} 
            \leq  \Big( \tilde{Q} +  \|\phi\|_{\infty} (\wt_+ - \wt_-) \Big) \frac{M \wt_+ \|\phi\|_{\infty}}{c(\wt_+ \|\phi\|_{L^1} - \tilde{Q})} 
    \end{equation*}
\end{theorem}

\begin{remark}
    In the Cucker-Smale case $\wt = 1$, the quantity $q = e/\rho$ is transported. 
    Consequently, $\tilde{Q} = \|q_0\|_{\infty}$
    and we recover the result of Leslie and Shvydkoy in \cite{LS-entropy}. 
    They observed that the bounding expression is linear in $\|q_0\|_{\infty}$ for small values 
    of $\|q_0\|_{\infty}$, showing that small initial values for $q_0$ lead to close 
    to uniform distributions of the flock for large times.  In the case presented here, 
    where $\wt$ is not necessarily equal to $1$, we pick up an additional term with 
    linear dependence on $(\wt_+ - \wt_-)$.  So, to obtain close to uniform 
    distributions of the flock for large times, one requires both smallness $\tilde{Q}$ and 
    smallness of $(\wt_+ - \wt_-)$.  See Remark \ref{rmk:bd_on_q_tilde} for smallness 
    conditions on $\tilde{Q}$. 
\end{remark}

\subsection{Outline of Paper}
The paper will be organized as follows.  
In section \ref{Properties}, we will discuss the basic properties of the $\st$-model-- namely, the maximum principle and the lack of both momentum conservation and an energy law--and compare these properties to that of the Cucker-Smale and Motsch-Tadmor models.  
In section \ref{alignment}, we record the $L^{\infty}$-based exponential alignment result and establish the conditional $L^2$-based alignment result under a local communication kernel, Theorem \ref{thm:l2_alignment_intro}. 
In section \ref{LWP}, local in time well-posedness along with the continuation criterion is established, Theorem \ref{lwp}.
In section \ref{GWP_1D}, 1D global well-posedness is established under the 
threshold criterion: $e_0 = \partial_x u_0 + \wt_0(\rho_0)_{\phi} \geq 0$, i.e. the 1D version of Theorem \ref{thm:UniGWP}.
In section \ref{sec:limiting_density_profile}, we establish, in 1D, strong flocking and estimate the deviation of the limiting density from the uniform distribution via relative entropy estimates, Theorem \ref{thm:strong_flocking_intro} and Theorem \ref{thm:entropy_estimate_intro}.
In section \ref{UniGWP}, the 1D global well-posedness argument is extended to uni-directional flocks in multiple dimensions, i.e. the full version of Theorem \ref{thm:UniGWP}.  
In Section~\ref{description_of_numerics}, we provide a comparison of numerical solutions to the $\st$-model, Motsch-Tadmor model, and Cucker-Smale model and a description of the numerical method. 
The numerics illustrate that when $\wt_0 = 1/\rho_{\phi}$, the $\st$-model displays similar qualitative behavior to the Motsch-Tadmor model, see tables \ref{plots:solns_CS}, \ref{plots:solns_WM_MT}, \ref{plots:solns_MT}.  
Convergence plots as the mesh size approaches zero are also included in order to validify the numerical method, see table \ref{plots:vary_k_and_h}.  Lastly, the Appendix contains some of the technical estimates used throughout the paper.

\section{Properties of $\st$-model}
\label{Properties}
The alignment force in \eqref{EAS_WM} has a similar structure to the alignment force in \eqref{EAS_CS} and thus inherits 
similar features. For instance, it inherits the maximum principle for the velocity (i.e. $\|\u\|_{\infty} \leq \|\u_0\|_{\infty}$) 
which is crucial for alignment. 
Nonetheless, there are some key differences.  The presence of the weight $\wt$ destroys the symmetry of the communication strength, 
as in the Motsch-Tadmor case.  That is to say, the force exerted by particle 'x' on particle 'y' may not be equal to the force exerted by particle 
'y' on particle 'x'.  As a consequence, there is no conservation of momentum nor is there a dissipative energy law.  
Let us illustrate how the weight $\wt$ obstructs these laws. 
The momentum $P$ and the energy $\cE$ are given by 
\begin{align*}
    P = \int_{\T} \rho \u \dx, \hspace{5mm} \cE = \frac{1}{2} \int_{\T} |\u|^2 \rho dx 
\end{align*}
For the momentum, we have  
\begin{align*}
    \frac{d}{dt} P
    &= - \int_{\T} \u \nabla \cdot (\u \rho) \dx + \int -\rho \u \cdot \nabla \u + \rho \wt ((\u\rho)_{\phi} - \u \rho_{\phi}) \hspace{1mm} \dx \\
    &= - \int_{\T} div(\rho |\u|^2) + \int \rho \wt ((\u\rho)_{\phi} - \u \rho_{\phi}) \hspace{1mm} \dx \\
    &= \int_{\T^2} \wt(x) (\u(y) - \u(x)) \phi(x-y) \rho(x) \rho(y) \hspace{1mm} \dy \dx \\
\end{align*}
Symmetrizing,  
\begin{align*}
    &= -\int_{\T^2} \wt(y)  (\u(y) - \u(x)) \phi(x-y) \rho(x) \rho(y) \hspace{1mm} \dy \dx \\
\end{align*}
and averaging the last two lines, we obtain 
\begin{equation}
    \label{momentum_time_deriv}
    \frac{d}{dt} P 
        = -\frac{1}{2} \int_{\T^2} (\u(x) - \u(y)) (\wt(x) - \wt(y)) \phi(x-y) \rho(x) \rho(y) \hspace{1mm} \dy \dx 
\end{equation}
With the presence of a non-constant $\wt$, we cannot conclude that this is equal to zero.  A similar computation 
shows that the same problem persists for the energy. 
\begin{align*}
    \frac{d}{dt} \cE 
        &= \int_{\T} \rho \u \wt \cdot (\phi * (\rho \u) - \u \phi * \rho) \hspace{1mm} \dx \\
        &= \int_{\T} \int_{\T} \u(x) \wt(x) \cdot (\u(y) - \u(x)) \phi(x-y) \rho(x) \rho(y) \hspace{1mm} \dy \dx \\
        &= -\int_{\T} \int_{\T} \u(y) \wt(y) \cdot (\u(y) - \u(x)) \phi(x-y) \rho(x) \rho(y) \hspace{1mm} \dy \dx \\
\end{align*}
Averaging the last two lines, we get 
\begin{align*}
    &= -\frac{1}{2} \int_{\T} \int_{\T} \big( \u(x) \wt(x) - \u(y) \wt(y) \big) \cdot \big(\u(y) - \u(x) \big) \phi(x-y) \rho(x) \rho(y) \hspace{1mm} \dy \dx \\
\end{align*}
The presence of the weight $\wt$ prevents us from guaranteeing pure dissipation. 
However, we can still decompose the law into a dissipative term and an extra term. 
We have 
\begin{align}
    \label{energy_time_deriv}
    \frac{d}{dt} \cE &= -\frac{1}{2} \int_{\Omega^2} \wt(x) |\u(x)-\u(y)|^2 \phi(x-y) \rho(x) \rho(y) \dx \dy \\
        &-\frac{1}{2} \int_{\Omega} \u(y) (\u(x) - \u(y)) (\wt(x) - \wt(y)) \phi(x-y) \rho(x) \rho(y) \dx \dy   \nonumber
\end{align} 
At first, the lack of a dissipative energy law looks to pose an obstruction for alignment since energy decay 
is connected to $L^2$-based alignment.  However, we are placated by the following two facts:
\begin{enumerate} [label = (\roman*)]
    \item The maximum principle still holds for the velocity equation. 
    In particular, $\|\u\|_{\infty} \leq \|\u_0\|_{\infty}$.  As a result, the $L^{\infty}$-based 
    alignment results can be established.  See Section \ref{linf_alignment} for the precise statement. 
    \item A dissipative energy law can be 
    recovered provided there are constraints on $\wt$ and $\rho$, which allow the second 
    term to absorbed into the dissipative term. 
    With an energy law at hand, we recover a conditional alignment result in the case 
    of local kernels. The details are discussed in Section \ref{local_alignment}.
\end{enumerate}
We summarize the differences between \eqref{EAS_CS}, \eqref{EAS_MT}, and \eqref{EAS_WM} in table \ref{table:compare_models} below.

\medskip
\begin{tabular}
    { | m{3em} | m{6em}| m{8em} | m{7em} | m{9em} | } 
    \hline 
       & Entropy Law & Performs well in Heterogeneous Formations & Conservation of Momentum & Energy dissipation \\
    \hline
    CS & \cmark & \xmark & \cmark & \cmark\\ 
    \hline
    MT & \xmark & \cmark & \xmark & \xmark \\ 
    \hline
    WM & \cmark & \cmark & \xmark & \xmark \\ 
    \hline
\end{tabular}
\label{table:compare_models}

\section{Alignment} 
\label{alignment}

\subsection{Alignment in $L^{\infty}$}
\label{linf_alignment}
The alignment $L^{\infty}$-based alignment result follows a similar Lyapunov-based approach 
to the Cucker-Smale case, given by Ha and Liu in \cite{HL2009}.  The presence of the weight does not introduce any difficulties as long as it's bounded 
away from zero. Indeed, the Lagrangian formulation of the velocity equation of \eqref{EAS_WM} is given by 
\begin{equation*}
    \begin{cases}
        \dot{X}(t, \alpha) = \V(t, \alpha) \\
        \dot{\V}(\alpha, t) = \wt(t, X(t,\alpha)) \int_{\T^n} \phi(X(t, \alpha) - X(t, \beta)) (\V(t,\alpha) - \V(t, \beta)) \rho_0(\beta) d\beta
    \end{cases}
\end{equation*}
We now state the alignment theorem without proof.
Let $\mathcal{A}(t) = \max_{\alpha, \beta \in \T^n} |\V(\alpha,t) - \V(\beta,t)|$.
\begin{theorem} [Alignment and Flocking)]
    \label{thm:linf_alignment}
    Assume (A1)-(A4). If in addition $\phi$ is bounded below, $\phi \geq c_1 > 0$ (which implies (A5)), then
    \begin{equation*}
        \mathcal{A}(t) \leq \cA_0 e^{-\wt_- M c_1 t}
    \end{equation*}  
\end{theorem}

\begin{remark}
    When $\Omega = \R^n$, the Theorem \ref{thm:linf_alignment} also holds.  In addition, the diameter $\mathcal{D}(t) = \max_{\alpha, \beta \in \T^n} |X(\alpha,t) - X(\beta,t)|$ of the flock 
    remains bounded. 
\end{remark}
Flocking then reduces to showing global well-posedness. 

\subsection{Alignment in $L^2$}
\label{local_alignment}
We now turn to the case of local kernels and the proof of Theorem \ref{thm:l2_alignment_intro}.
For local kernels, the communication strength vanishes for agents
that are more than a distance $r_0$ apart.  In other words, $\phi(x,y) \geq c_1 \mathbbm{1}_{|x-y|<r_0}$. 
In this case, the $L^{\infty}$-based arguments in Section \ref{linf_alignment} fail due to the 
lack of a lower bound on the alignment force.
However, there is $V_2$-based alignment provided there is energy dissipation, which is present as long as \eqref{ineq:dissipation_condition} holds.
\begin{proof}[Proof of Theorem \ref{thm:l2_alignment_intro}]
    Let $\cE(t)$ and $P(t)$ be the energy and momentum as defined in Section \ref{Properties}. 
    Observe that $V_2 = \cE(t) - \frac{1}{2M} P(t)^2$. Using this along with the momentum and energy equations \eqref{momentum_time_deriv} and \eqref{energy_time_deriv}, 
    we obtain 
    \begin{align*}
        \frac{d}{dt} V_2 &= \frac{d}{dt} \cE - \frac{1}{M} P \frac{d}{dt} P \\
            &= -\frac{1}{2} \int_{\T^{2n}} \wt(x) |\u(x)-\u(y)|^2 \phi(x-y) \rho(x)\rho(y) \dx \dy \\
            &- \frac{1}{2} \int_{\T^{2n}} \u(y) (\u(x)-\u(y)) (\wt(x)-\wt(y)) \phi(x-y) \rho(x) \rho(y) \dx \dy \\ 
            &+ \frac{1}{2M} \int_{\T^n} \u(x) \rho(x) \dx \int_{\T^{2n}} (\u(x)-\u(y)) (\wt(x)-\wt(y)) \phi(x-y) \rho(x)\rho(y) \dx \dy  \\
    \end{align*}
    With the goal relating it back to $V_2$, we write it as 
    \begin{align*}
            &= -\frac{1}{2} \int_{\T^{2n}} \wt(x) |\u(x)-\u(y)|^2 \phi(x-y) \rho(x)\rho(y) \dx \dy \\
            &- \frac{1}{2} \int_{\T^{2n}} (\u(y) - \bar{\u}) (\u(x) - \bar{\u} - (\u(y) - \bar{\u})) (\wt(x)-\wt(y)) \phi(x-y) \rho(x) \rho(y) \dx \dy \\ 
            &:= -I_1 + I_2 
    \end{align*} 
    The dissipation term $I_1$ can be bounded from below by
    \begin{equation*}
        I_1 \geq c_0 \wt_- \rho_-^2 \int_{|x-y| \leq r_0} |\u(x) - \u(y)|^2 \dx \dy 
    \end{equation*}
    Let $Ave(\u) = \frac{1}{(2\pi)^n} \int_{\T^n} \u(x) \dx$.
    Using Lemma 2.1 of \cite{LS-entropy} we obtain for some constant $c_1' := c_1'(r_0, c_1)$, 
    \begin{equation*}
            \geq c_1' \wt_- \rho_-^2 \frac{1}{2} \int_{\T^n} |\u(x) - Ave(\u)|^2 \dx 
    \end{equation*}
    And from the identity $\int_{\T^n} |\u(x) - Ave(\u)|^2 \rho(x) \dx = \int_{\T^n} |\u(x) - \bar{\u}|^2 \rho(x) \dx + \int_{\T^n} |Ave(\u) - \bar{\u}|^2 \rho(x) \dx$, 
    \begin{align*}
        &\geq c_1' \wt_- \frac{\rho_-^2}{\rho_+} \frac{1}{2} \int_{\T^n} |\u(x) - \bar{\u}|^2 \rho(x) \dx \\
        &= c_1' \wt_- \frac{\rho_-^2}{\rho_+} V_2 \\            
    \end{align*}
    In $I_2$, the first term vanishes due to symmetrization.  So we are left with 
    \begin{align*}
        |I_2| 
            &\leq \frac{1}{2}  \int_{\T^{2n}} |\u(y) - \bar{\u}|^2 |\wt(x)-\wt(y)| \phi(x-y) \rho(x) \rho(y) \dx \dy \\
            &\leq \frac{1}{2} (\wt_+ - \wt_-) \|\phi\|_{\infty} \int_{\T^2} |\u(y) - \bar{\u}|^2 \rho(x) \rho(y) \dx \dy \\
            &= \frac{M}{2} (\wt_+ - \wt_-) \|\phi\|_{\infty} V_2 \\
    \end{align*}
    To absorb $I_2$ into the dissipative term $I_1$, we require that 
    \begin{equation*}
        \frac{M}{2} (\wt_+ - \wt_-) \|\phi\|_{\infty} \leq \frac{1}{2} c_1' \wt_- \frac{\rho_-^2}{\rho_+} 
    \end{equation*}
    which is equivalent to the condition \eqref{ineq:dissipation_condition}. 
    Under this constraint, we have 
    \begin{equation*}
        \frac{d}{dt} V_2 \leq - \frac{1}{2} c_1' \wt_- \frac{\rho_-^2}{\rho_+} V_2 
    \end{equation*}
    Owing to the uniform lower bound on $\frac{\rho_-^2}{\rho_+}$, we obtain exponential decay of $V_2$. 
\end{proof}

\section{Local Well-Posedness}
\label{LWP}
We will prove Theorem \ref{lwp} and establish global well-posedness for small initial data, Theorem \ref{thm:small_data_intro}.

\subsection{Local well-posedness for a viscous regularization}
The strategy is to obtain a local well-posedness result for a viscous regularization of \eqref{EAS_WM} given by 
\begin{align*}
    \label{regularized_EAS_WM}
    \tag{WM'}
    \begin{cases}
        \partial_t \rho + \nabla \cdot (\u\rho) = \epsilon \Delta \rho \\
        \partial_t \wt + \u_F \cdot \nabla \wt = \epsilon \Delta \wt \\
        \partial_t \u + \u \cdot \nabla \u = \wt ((\u \rho)_{\phi} - \u \rho_{\phi}) + \epsilon \Delta \u
    \end{cases}
\end{align*}
Then we will show that the time of existence for the regularized equation \eqref{regularized_EAS_WM} does not depend on $\epsilon$ and moreover that there's a subsequence of solutions converging to solutions to \eqref{EAS_WM}. 
The proof of local well-posedness of \eqref{regularized_EAS_WM} uses the standard Banach Fixed Point argument. 
Let $Z(t,x) = (\rho(t,x), \wt(t,x), \u(t,x))$ and let $\mathcal{N}(Z)$ denote the non-linear terms in \eqref{regularized_EAS_WM}. 
Define the map 
\begin{equation}
    \label{duhamel_formula}
    \mathcal{F}[Z](t) = 
        e^{\epsilon t \Delta} Z_0 + \int_0^t e^{\epsilon (t-s) \Delta} \mathcal{N}(Z(s)) ds 
\end{equation}
Recall that $c_0$ is the lower bound on $\rho_{\phi}$ given in assumption (A5), i.e. $\rho_{\phi} \geq c_0$. 
Let $r = c_0 / (2 \|\phi\|_{L^{\infty}})$.
We will show that there exists a small time $T$ so that this map is a contraction mapping 
on $C([0,T); B_{r}(Z_0))$ where $B_{r}(Z_0)$ denotes the the ball of radius $r$ in 
$X = (H^k \cap L^1_+) \times H^l \times H^m$ centered at $Z_0$.
The choice of the $r$ guarantees that 
\begin{equation*}
    \|\rho_{\phi}\|_{\infty} \geq c_0 - \|(\rho - \rho_0)_{\phi}\|_{\infty}
        \geq c_0 - \|\rho - \rho_0\|_{L^1} \|{\phi}\|_{L^{\infty}} \geq c_0/2 
\end{equation*}
so the lower bound $\|\rho_{\phi}\|_{\infty} \geq c_0/2$ automatically holds.
Invariance and contractivity of the map $\mathcal{F}$ will be obtained 
from estimates on $\|\int_0^t e^{\epsilon(t-s)\Delta} \mathcal{N}(Z(s)) ds\|_X$.

In the estimates below, we will use the following notation. 
If $U$ and $V$ are quantities, then $U \lesssim V$ is equivalent to 
$U \leq C(n, c_0, l, M, \phi, \epsilon) V$, where $M$ is the mass. Importantly, the constant $C$ 
does not depend on the time $T$. 
The non-linear estimates on derivatives from Appendix \ref{appdx:non-linear_estimates} 
will be used in order to estimate Sobolev norms of the Favre Filtration. 
We now proceed to show invariance of the map $\mathcal{F}$ by showing that
\begin{equation*}
    \|\mathcal{F}(Z(s)) - Z_0\|_X \leq 
        \Big\| e^{\epsilon t \Delta} Z_0 - Z_0 \Big\|_X + \Big\| 
            \int_0^t e^{\epsilon (t-s) \Delta} \mathcal{N}(Z(s)) ds \Big\|_X 
            \leq 1
\end{equation*}
provided $T$ is small. 
The first term can be made small by the continuity property of the heat semigroup. 
Regarding the second term, the $\rho$-equation has been estimated in \cite{Sbook}. 
For the $\wt$ equation, we use the Analyticity property of the heat equation 
and $l \geq n/2+2$ to get   
\begin{align*}
    \Big\| \partial^l 
        &\int_0^t e^{\epsilon(t-s) \Delta} \u_F \cdot \nabla \wt ds \Big\|_2
            \leq \int_0^t \frac{1}{\sqrt{\epsilon(t-s)}} \big\| \partial^{l-1} ( \u_F \cdot \nabla \wt) \big\|_2 ds \\
            &\leq \frac{T^{1/2}}{\epsilon^{1/2}} \Big( \|\nabla \wt\|_{\infty} \|\uavg\|_{\dot{H}^{l-1}} + \| \wt \|_{\Hldot} \|\uavg\|_{\infty} \Big) \\
\end{align*}
Using \ref{favre_estimate} to estimate $\|\uavg\|_{\dot{H}^{l-1}}$, we get 
\begin{align*}
    &\lesssim \frac{T^{1/2}}{\epsilon^{1/2}} 
         \Big( 
            \|\nabla \wt\|_{\infty}  \Big( \Big\| \frac{1}{\rho_{\phi}} \Big\|_{\infty} \|(\u \rho)_{\phi} \|_{H^{l-1}}
            + \Big\| \frac{1}{\rho_{\phi}} \Big\|_{\infty}^{l} \| \rho_{\phi} \|_{H^{l-1}} \|(\u \rho)_{\phi} \|_{\infty} \Big)
            + \| \wt \|_{\dot{H}^{l}} \|\u\|_{\infty} \Big) \\
    &\lesssim \frac{T^{1/2}}{\epsilon^{1/2}} 
        \Big( 
            \|\nabla \wt\|_{\infty}  \|\u\|_{\infty} 
            + \| \wt \|_{\dot{H}^{l}} \|\u\|_{\infty} \Big) \\
    &\lesssim \frac{T^{1/2}}{\epsilon^{1/2}} 
            \| \wt \|_{\dot{H}^{l}} \|\u\|_{\infty} \\
\end{align*}
This quantity is small for small $T$. 
We only needed $l \geq n/2 + 1$ here, but 
the energy estimates later will impose a more strict condition on the exponent. 
For the velocity equation, the transport term is estimated in \cite{Sbook}. It remains to estimate the alignment term.  
\begin{align*}
    \Big\| \partial^m
        &\int_0^t e^{\epsilon(t-s) \Delta} \wt \big( (\u\rho)_{\phi} - \u \rho_{\phi} \big) ds \Big\|_2
            \leq \int_0^t \frac{1}{\sqrt{\epsilon(t-s)}} \big\| \partial^{m-1} \big( \wt \big( (\u\rho)_{\phi} - \u \rho_{\phi} \big) \big\|_2 ds \\
            &\leq \frac{T^{1/2}}{\epsilon^{1/2}} \Big( \| \wt \|_{\infty} \|(\u\rho)_{\phi} - \u \rho_{\phi}\|_{\dot{H}^{m-1}} + \| \wt \|_{\dot{H}^{m-1}} \|(\u\rho)_{\phi} - \u \rho_{\phi}\|_{\infty} \Big) \\
            &\lesssim \frac{T^{1/2}}{\epsilon^{1/2}} \Big( \| \wt \|_{\infty} \|\u\|_{\dot{H}^{m-1}} + \| \wt \|_{\dot{H}^{m-1}} \|\u\|_{\infty} \Big) \\
        \end{align*}
This expression is small for small $T$.  
To show that $\mathcal{F}$ is a contraction, we will to show for $Z_1, Z_2 \in C([0,T]; B_r(Z_0))$
that 
\begin{equation*}
    \| \mathcal{F}(Z_1) - \mathcal{F}(Z_2) \|_{C([0,T]; B_{r}(Z_0))} \leq \alpha \|Z_1 - Z_2\|_{C([0,T]; B_{r}(Z_0))}
\end{equation*}
for some $0 < \alpha < 1$ where $\|f\|_{C([0,T]; B_{r}(Z_0))} = \sup_{0 \leq s \leq T} \|f(s)\|_X$.
In the following contractivity estimates, $C = C(n, c_0, l, M_1, M_2, \epsilon, \phi)$ is a constant where 
$M_1$, $M_2$  denotes the mass of the two respective solutions. The constant $C$ may change in each line; 
and $[\u_i]_{\rho_j} := (\u_i \rho_j)_{\phi}/(\rho_j)_{\phi}$ denotes the Favre averaging
associated to $\u_i$, $\rho_j$.
For the $\wt$-equation, using the algebra property of the $H^l$ norm, we obtain
\begin{align*}
    \Big\| \partial^l 
        &\int_0^t e^{\epsilon(t-s) \Delta} \Big( 
            [\u_1]_{\rho_1} \cdot \nabla \wt_1 -  [\u_2]_{\rho_2} \cdot \nabla \wt_2 \Big) ds \Big\|_2 \\
        &\lesssim \Big\| \partial^l \int_0^t e^{\epsilon(t-s) \Delta} \Big( 
            \big( [\u_1]_{\rho_1} - [\u_1]_{\rho_2} \big) \cdot \nabla \wt_1 + [\u_1 - \u_2]_{\rho_2} \cdot \ 
            \nabla \wt_1 + [\u_2]_{\rho_2} \cdot \nabla \big(\wt_1 - \wt_2 \big) \Big) ds \Big\|_2 \\
        &\lesssim \frac{T^{1/2}}{\epsilon^{1/2}} \Big\| 
            \big( [\u_1]_{\rho_1} - [\u_1]_{\rho_2} \big) \cdot \nabla \wt_1 + [\u_1 - \u_2]_{\rho_2} \cdot \ 
            \nabla \wt_1 + [\u_2]_{\rho_2} \cdot \nabla \big(\wt_1 - \wt_2 \big) \Big\|_{\dot{H}^{l-1}} \\
        &\lesssim \frac{T^{1/2}}{\epsilon^{1/2}} \Big(  
            \| [\u_1]_{\rho_1} - [\u_1]_{\rho_2} \|_{\dot{H}^{l-1}} \| \wt_1 \|_{\dot{H}^{l}} \  
            + \| [\u_1 - \u_2]_{\rho_2} \|_{\dot{H}^{l-1}} \cdot \| \wt_1\|_{\dot{H}^{l}} \\
            &+ \| [\u_2]_{\rho_2}\|_{\dot{H}^{l-1}}  \| \wt_1 - \wt_2 \|_{\dot{H}^{l}} \Big) \\
\end{align*}
By \ref{favre_estimate}, we have 
\begin{align*}
    \|[\u_2]_{\rho_2}\|_{\dot{H}^{l-1}} 
        &\leq \Big\| \frac{1}{{\rho_2}_{\phi}} \Big\|_{\infty} \|(\u_2 \rho_2)_{\phi} \|_{H^{l-1}}
            + \Big\| \frac{1}{{\rho_2}_{\phi}} \Big\|_{\infty}^{l} \| {\rho_2}_{\phi} \|_{H^{l-1}} \|(\u_2 \rho_2)_{\phi} \|_{\infty} \\
        &\leq C \|\u_2\|_{\infty} 
\end{align*}
And by \ref{contractivity_estimate}, we have 
\begin{align*}
    \|[\u_1]_{\rho_1} - [\u_1]_{\rho_2}\|_{\dot{H}^{l}}  
        &\leq \Big\| \frac{1}{{\rho_1}_{\phi} {\rho_2}_{\phi}} \Big\|_{\infty} 
        \| (\u_1 \rho_1)_{\phi} ({\rho_2})_{\phi} - (\u_1 \rho_2)_{\phi} ({\rho_1})_{\phi} \|_{H^l} \\
    &+ \Big\| \frac{1}{{\rho_1}_{\phi} {\rho_2}_{\phi}} \Big\|_{\infty}^{l+1} \| {\rho_1}_{\phi} {\rho_2}_{\phi} \|_{H^l}
    \| (\u_1 \rho_1)_{\phi} ({\rho_2})_{\phi} - (\u_1 \rho_2)_{\phi} ({\rho_1})_{\phi} \|_{\infty} \\
    &\leq C \| (\u_1 \rho_1)_{\phi} ({\rho_2})_{\phi} - (\u_1 \rho_2)_{\phi} ({\rho_1})_{\phi} \|_{H^l} \\
    &= C \| ({\u_1 \rho_1- \u_1\rho_2})_{\phi} ({\rho_2})_{\phi} + (\u_1 {\rho_2})_{\phi} ({\rho_2 - \rho_1})_{\phi} \|_{H^l} \\
    &\leq C \|\u_1\|_{\infty} \| Z_1 - Z_2 \|_X  \\
\end{align*}
All together, we obtain 
\begin{align*}
    \Big\| \partial^l 
        &\int_0^t e^{\epsilon(t-s) \Delta} \Big[ 
            [\u_1]_{\rho_1} \cdot \nabla \wt_1 -  [\u_2]_{\rho_2} \cdot \nabla \wt_2 \Big] ds \Big\|_2 \\
        &\leq C \frac{T^{1/2}}{\epsilon^{1/2}} \Big[
            \| Z_1\|_X \|Z_1 - Z_2\|_{X} + \|Z_2\|_X \| Z_1 - Z_2 \|_X 
         \Big] \\
        &\leq C \frac{T^{1/2}}{\epsilon^{1/2}} 
            \|Z_0 + r\|_X \| Z_1 - Z_2 \|_X  \\
\end{align*}
The contractivity for the $\rho$ and $\u$-equations can be estimated similarly.
This shows $\mathcal{F}$ is a contraction mapping for some small enough time of existence, $T$. 

\subsection{Time of existence is independent of $\epsilon$ and energy estimates}
\label{energy_estimates}
We will now show that $T$ does not depend on $\epsilon$ by obtaining 
$\epsilon$-independent energy estimates on the norm of the solution 
\begin{equation*}
    Y_{m,l,k} = \|\u\|_{H^m}^2 + \| \wt \|_{H^l}^2 + \|\rho\|_{H^k}^2 + \|\rho\|_1^2
\end{equation*}
The mass $\|\rho\|_1$ is conserved so it remains to control the other terms. 
The $\epsilon$-independent energy estimate for the $\rho$ equation is given in \cite{Sbook}. 
We record it here. Provided $m \geq k+1$,
\begin{equation*}
    \frac{d}{dt} \|\rho\|_{\Hkdot}^2 \leq C (\|\nabla \u\|_{\infty} + \|\nabla \rho\|_{\infty} + \|\rho\|_{\infty}) Y_{m,l,k}
\end{equation*}
For the purpose of obtaining a continuation criterion in Section \ref{sec_cont_criterion}, we will explicitly include the dependence of the energy estimates on $\|1/\rho_{\phi}\|_{\infty}$ 
instead of absorbing it into the implied constant. 
Testing the $\wt$-equation with $\partial^{2l} \wt$, we get 
\begin{equation*}
    \frac{d}{dt} \| \wt \|_{\Hldot}^2 = \int_{\T^n} \u_F \cdot (\nabla \wt) \partial^{2l} \wt \dx - \epsilon \| \wt \|_{\dot{H}^{l+1}}^2
\end{equation*}
Integrating by parts and using the commutator estimate and \ref{favre_estimate}, we get 
\begin{align*}
    \frac{d}{dt} \| \wt \|_{\Hldot}^2  
        &= \frac{1}{2} \int_{\T^n} \nabla \cdot \u_F |\partial^l \wt|^2 \dx   
            - \int_{\T^n} \Big( \partial^l ( \u_F \cdot \nabla \wt ) - \u_F \cdot \nabla \partial^l \wt \Big) \partial^l \wt \dx 
            - \epsilon \| \wt \|_{\dot{H}^{l+1}}^2 \\
        &\lesssim \|\nabla \u_F\|_{\infty} \| \wt \|_{\Hldot}^2   
            + \big( \|\u_F\|_{\Hldot} \|\nabla \wt\|_{\infty} + \|\nabla \u_F\|_{\infty} \| \wt \|_{\Hldot} \big) \| \wt \|_{\Hldot} \\
        &\lesssim \Big\| \frac{1}{\rho_{\phi}} \Big\|_{\infty}^2 \|\u\|_{\infty} \| \wt \|_{\Hldot}^2    
            + \Big( \Big\| \frac{1}{\rho_{\phi}} \Big\|_{\infty}^{l+1} \|\u\|_{\infty} \|\nabla \wt\|_{\infty} 
            + \Big\| \frac{1}{\rho_{\phi}} \Big\|_{\infty}^2 \|\u\|_{\infty} \Big) \| \wt \|_{\Hldot} \\
        &\lesssim \Big\| \frac{1}{\rho_{\phi}} \Big\|_{\infty}^{l+1} \|\u\|_{\infty} \| \wt \|_{\Hldot}^2    
\end{align*}
Testing the velocity equation with $\partial^{2m} \u$, and using the commutator estimate, we get 
\begin{equation*}
    \int_{\T^n} \partial^m (\u \cdot \nabla \u ) \partial^m \u \dx \lesssim \|\nabla \u\|_{\infty} \|\u\|_{\Hmdot}^2
\end{equation*}
For the alignment terms, we use the product estimate to get  
\begin{align*}
    \int_{\T^n} \partial^m \Big( \wt ( (\u\rho)_{\phi} - \rho_{\phi} ) \Big) \partial^m \u \dx 
        &\leq \big(\| \wt \|_{\infty} \|(\u\rho)_{\phi} - \rho_{\phi}\|_{\Hmdot} 
            + \| \wt \|_{\Hmdot} \|(\u\rho)_{\phi} - \rho_{\phi} \|_{\infty} \big) \|\u\|_{\Hmdot}  \\
        &\lesssim (\| \wt \|_{\infty} + \| \wt \|_{\Hmdot}) \|\u\|_{\infty} \|\u\|_{\Hmdot} 
\end{align*}
Provided $l \geq m$, the energy estimate for the velocity equation becomes, 
\begin{equation*}
    \frac{d}{dt} \|\u\|_{\Hmdot}^2 \leq (\|\nabla \u\|_{\infty} + \| \wt \|_{\infty} \|\u\|_{\infty} + \|\u\|_{\infty}) Y_{m,l,k}
\end{equation*}
Combining the energy estimates, we get an $\epsilon$-independent estimate on the norm $Y_{m,l,k}$.
\begin{align*}
    \frac{d}{dt} Y_{m,l,k} &\leq C \Big(\|\nabla \u\|_{\infty} + \|\nabla \rho\|_{\infty} + \|\rho\|_{\infty} 
        + \Big\| \frac{1}{\rho_{\phi}} \Big\|_{\infty}^{l+1} \|\u\|_{\infty} 
        + (1 + \|\wt\|_{\infty}) \| \u \|_{\infty} \Big) Y_{m,l,k} 
\end{align*}
Recall that $\rho_{\phi} \geq c_0/2$ up to some $\epsilon$-independent time $T'$ and that $\wt$ is bounded uniformly in time.
So, for $l \geq m \geq k+1 > n/2 + 2$, we obtain
\begin{equation*}
    \frac{d}{dt} Y_{m,l,k} \lesssim Y_{m,l, k}^{2},  \hspace{5mm} \text{ for } t < T'
\end{equation*}
This gives a bound on $Y_{m,l,k}$ up to some positive time $T>0$.
Due to the $\epsilon$-independence of these energy estimates, 
the local solutions to \eqref{regularized_EAS_WM} exist on the common time 
interval $[0,T]$ independent of $\epsilon$.  
We will conclude by taking $\epsilon \to 0$ 
in \eqref{regularized_EAS_WM} and using compactness properties to obtain 
a local in time solution to \eqref{EAS_WM}.

\begin{remark}
    \label{est_on_dissipative_terms}
    The full energy estimate including the dissipative terms show that for some 
    $\epsilon$-independent constant $C$, 
    \begin{equation*}
        \epsilon \int_0^T \big( \|\rho\|_{H^{k+1}}^2 + \| \wt \|_{H^{l+1}}^2 + \|\u\|_{H^{m+1}}^2 \big) ds \leq C 
    \end{equation*}
\end{remark}

\subsection{Viscous solutions to \eqref{regularized_EAS_WM} approach solutions to \eqref{EAS_WM}}
The following estimate on the time derivative will yield the necessary compactness properties. 
We will denote the solution to \eqref{regularized_EAS_WM} by $Z^{\epsilon}$ instead 
of just $Z$ in order to emphasize the dependence of the solution on $\epsilon$. 
From squaring the time derivatives in \eqref{regularized_EAS_WM}, we see that 
\begin{equation*}
    \|\partial_t Z^{\epsilon}\|_{L^2}^2 \lesssim \|Z^{\epsilon}\|_X^2 + \epsilon \|Z^{\epsilon}\|_{H^2}^2
\end{equation*}
From this inequality and Remark \ref{est_on_dissipative_terms}, 
we obtain $\partial_t Z^{\epsilon} \in L^2([0,T]; L^2 \times L^2 \times L^2)$. 
Of course, we also have $Z^{\epsilon} \in L^{\infty}([0,T]; X)$ by local well-posedness. 
Letting $Y = H^{k-1} \times H^{l-1} \times H^{m-1}$ and applying the Aubin-Lions Lemma, we 
obtain a convergent subsequence $Z^{\epsilon} \to Z^*$ in $C([0,T]; Y)$ for some $Z^*$. 
That $Z^* = Z$, the solution to \eqref{EAS_WM}, can be seen by taking 
the limit as $\epsilon \to 0$ in the Duhamel formula \eqref{duhamel_formula}. 
Indeed, since $l \geq m \geq k+1 \geq n/2 + 3$, we have the pointwise convergence 
$\mathcal{N}(Z^{\epsilon}) \to \mathcal{N}(Z^*)$ so by the dominated convergence 
theorem and the continuity property of the heat semigroup, we conclude that 
\begin{equation*}
    \mathcal{F}[Z^*](t) = Z^*_0 + \int_0^t \mathcal{N}(Z^*(s)) ds 
\end{equation*}
which is the solution to \eqref{EAS_WM}.  
Further, since $Z^* = Z \in C([0,T]; Y)$ and $Y$ is dense in $H^{-k} \times H^{-l} \times H^{-m}$, 
the solution is weakly continuous, i.e. $Z \in C_w([0,T]; X)$. This concludes the 
local in time existence and uniqueness part of Theorem \ref{lwp}.

\subsection{Conditions for Continuation of the Solution}
\label{sec_cont_criterion}
Let us now establish the continuation criterion \eqref{cont_criterion}, which we will use to prove
conditional global existence for unidirectional flocks in Sections \ref{GWP_1D} and \ref{UniGWP}.  Recall 
the relevant energy estimates from Section \ref{energy_estimates}. 
\begin{align}
    \frac{d}{dt} &\|\rho\|_{\Hkdot}^2 \leq C (\|\nabla \u\|_{\infty} + \|\nabla \rho\|_{\infty} + \|\rho\|_{\infty}) Y_{m,l,k}  \label{rho_energy_estimate} \\
    \frac{d}{dt} &\| \wt \|_{\Hldot}^2 \lesssim \Big\| \frac{1}{\rho_{\phi}} \Big\|_{\infty}^{l+1} \|\u\|_{\infty} \| \wt \|_{\Hldot}^2 \label{w_energy_estimate} \\   
    \frac{d}{dt} &\|\u\|_{\Hmdot}^2 \lesssim \|\nabla \u\|_{\infty} \|\u\|_{\Hmdot}^2 + (\| \wt \|_{\infty} \|\u\|_{\infty} + \|\u\|_{\infty} \| \wt \|_{\Hmdot}) \|\u\|_{\Hmdot} \label{u_energy_estimate}
\end{align}
A sufficient continuation criterion is then given by 
\begin{equation*}
    \int_0^{T} \Big(\|\nabla \u\|_{\infty} + \|\rho\|_{\infty} + \|\nabla \rho\|_{\infty} +  \Big\|\frac{1}{\rho_{\phi}} \Big\|_{\infty}^{l+1}  \Big) ds < \infty
\end{equation*}
That is, the $X$-norm of the solution will not blow up for finite times provided this holds.  We can simplify this criterion 
to \eqref{cont_criterion} by showing that $\|\nabla \u\|_{\infty}$ and $\| 1/\rho_{\phi} \|_{\infty}$ 
control $\|\rho\|_{\infty}$ and $\|\nabla \rho\|_{\infty}$.  Indeed, assume that \eqref{cont_criterion} holds.
Then solving the continuity equation along characteristics, $\dot{X}(t,\alpha) = \u(t, X(t,\alpha))$, we get 
\begin{equation*}
    \rho(X(t, \alpha)) = \rho(0, \alpha) \exp \Big\{ -\int_0^T (\nabla \cdot \u) (s, X(s,\alpha)) ds \Big\}
\end{equation*}
From the integrability of $\|\nabla \u\|_{\infty}$ in \eqref{cont_criterion}, we obtain that $\|\rho\|_{\infty}$ is bounded. 
Further, boundedness of $\nabla \rho$ can be obtained similarly by differentiating the continuity equation.  
The equation for an arbitrary partial derivative $\partial \rho$ is given by
\begin{align*}
    \partial_t (\partial \rho) + \u \cdot \nabla \partial \rho + \partial \u \cdot \nabla \rho + (\nabla \cdot \u) \partial \rho + (\nabla \cdot \partial \u) \rho = 0
\end{align*}
Then 
\begin{equation*}
    \frac{d}{dt}\|\partial \rho\|_{\infty} 
        \leq \| \frac{d}{dt} \partial \rho \|_{\infty} 
        \leq \| \nabla \u\|_{\infty} \| \nabla \rho \|_{\infty} + \|\nabla \u\|_{\infty} \|\partial \rho\|_{\infty} + \| \nabla^2 \u\|_{\infty} \|\rho\|_{\infty} 
\end{equation*}
Summing over the partials, we get 
\begin{equation*}
    \frac{d}{dt}\|\nabla \rho\|_{\infty} 
        \lesssim \| \nabla \u\|_{\infty} \| \nabla \rho \|_{\infty} + \| \nabla^2 \u\|_{\infty} \|\rho\|_{\infty} 
\end{equation*}
To conclude by Gronwall, we need to show $\| \nabla^2 \u\|_{\infty}$ is bounded for finite times. 
Observe that $\| \nabla^2 \u\|_{\infty} \leq \|\u\|_{H^m}$ since $m > n/2 + 2$ so 
it suffices to bound $\|\u\|_{H^m}$. 
From the energy estimate \eqref{w_energy_estimate} on the $\wt$-equation, \eqref{cont_criterion} implies that $\| \wt \|_{H^l}$ is bounded
for finite times, and in turn the energy estimate \eqref{u_energy_estimate} on the $\u$-equation implies that $\|\u\|_{H^m}$ is 
bounded for finite times.

\subsection{Small Initial Data}
With the continuation criterion in hand, we will prove Theorem \ref{thm:small_data_intro}. The precise 
statement is given below in Theorem \ref{small_data}. 
Provided the kernel is bounded away from zero and the initial variations of $\u$ are sufficiently small, control 
of $\|\nabla \u\|_{\infty}$ can be established in any dimension.  The small variation of $\u$ allows 
the quadratic term in the equation for $\partial \u$ to be absorbed into the dissipative term. 
Intuitively speaking, the strength of the alignment force will overpower the Burger's transport term.  
Letting $\cA(t) = \max_{x,y \in \T} |\u(x) - \u(y)|$, we state the result.
\begin{theorem}
    \label{small_data}
    Assume (A1)-(A4). If in addition, the kernel $\phi$ is bounded below, $\phi \geq c_1 > 0$ (which implies (A5)), and the following smallness conditions: 
    \begin{align*}
        &\cA_0 < \epsilon^2, \hspace{5mm} \|\u_0\|_{\infty} < \epsilon, \hspace{5mm} \epsilon < \frac{c_1 \wt_- M}{2 + \eta  M \|\phi\|_{\infty} + \wt_+ M \|\nabla \phi\|_{\infty}}, \\
        &\eta = \|\nabla \wt_0\|_{\infty} \exp \Big\{ \frac{2 \|\phi\|_{\infty} \|\phi'\|_{\infty}}{M (\wt_-) c_1^3} \cA_0 \Big\}
    \end{align*}
    then there is a unique solution 
    $(\rho, \wt, \u) \in C_w([0,T]; (H^k \cap L^1_+) \times H^l \times H^m)$ to \eqref{EAS_WM}
    existing globally in time such that 
    \begin{equation*}
        \| \nabla \u \|_{\infty} < 2 \epsilon, \hspace{5mm} t > 0 
    \end{equation*}
\end{theorem}

\begin{remark}
    \label{remark:small_data}
    The essential ingredients for the proof are exponential alignment in $L^{\infty}$ and a bound from below on $\rho_{\phi}$, 
    from which we can obtain a lower bound on the dissipative term in the equation for $\partial \u$ and a uniform bound on $\|\nabla \wt\|_{\infty}$, 
    which in turn controls another term in the equation for $\partial \u$.
    The quantity $\eta$ above denotes this uniform bound on $\|\nabla \wt\|_{\infty}$.
    For general kernels, $\|\nabla \wt\|_{\infty}$ may not be bounded. 
\end{remark}
\begin{remark}
    The conclusion $\|\nabla \u\|_{\infty} < 2 \epsilon$ for $t > 0$ can be bootstrapped to obtain exponential decay of $\|\nabla \u\|_{\infty}$. 
\end{remark}

\begin{proof}
    By assumption, $\rho_{\phi} \geq c_1 M > 0$. Then the continuation criterion \eqref{cont_criterion} 
    reduces to control of $\|\nabla \u\|_{\infty}$. 
    The equation for an arbitrary partial derivative $\partial \u$ is given by 
    \begin{equation}
        \label{eqn:partial_u}
        \partial_t (\partial \u) + \u \cdot \nabla \partial \u = 
            \partial \wt ( (\u\rho)_{\phi} - \u \rho_{\phi} )
            + \wt ( (\u\rho)_{\phi'} - \u \rho_{\phi'} )
            - \nabla \u \cdot \partial \u - \wt \rho_{\phi} \partial \u
    \end{equation}
    Then 
    \begin{equation*}
        \frac{d}{dt} \|\partial \u\|_{\infty} 
            \leq \big( \| \partial \wt \|_{\infty} M \|\phi\|_{\infty} + \wt_+ M \|\nabla \phi\|_{\infty} \big) \cA(t) 
            + \big( \|\nabla \u\|_{\infty} - c_1 \wt_- M  \big) \|\partial \u \|_{\infty}
    \end{equation*}
    The equation for $\partial \wt$ is given by
    \begin{equation*}
        \partial_t (\partial \wt) + \u_F \cdot \nabla \partial \wt = \partial \u_F \cdot \nabla \wt 
    \end{equation*}
    Therefore, in order to bound $\|\partial \wt\|_{\infty}$, the exponential decay of $\|\partial \u_F\|_{\infty}$ is sufficient.
    Due to $\rho_{\phi} \geq c_1 M$ and exponential alignment (Theorem \ref{thm:linf_alignment}), we have 
    \begin{align}
        \label{ineq:decay_u_F}
        |\partial \u_F|
            &= \Big| \frac{\rho_{\phi} (\u \rho)_{\phi'} - \rho_{\phi'} (\u\rho)_{\phi} }{\rho_{\phi}^2} \Big|
            = \Big| \frac{\rho_{\phi} \big( (\u \rho)_{\phi'} - \u \rho_{\phi'} \big) 
                + \rho_{\phi'} \big( (\u\rho)_{\phi} - \u \rho_{\phi} \big) }{\rho_{\phi}^2} \Big| \\
            &\leq \frac{2 \|\phi\|_{\infty} \|\phi'\|_{\infty}}{c_1^2} \cA(t) \leq \frac{2 \|\phi\|_{\infty} \|\phi'\|_{\infty}}{c_1^2} \cA_0 e^{-\wt_- Mc_1t} \nonumber
    \end{align}
    As a result, integrating the $\partial \wt$ equation, we obtain
    \begin{equation}
        \label{eqn:bd_on_w_x}
        \|\nabla \wt\|_{\infty} \leq \|\nabla \wt_0\|_{\infty} \exp \Big\{ \frac{2 \|\phi\|_{\infty} \|\phi'\|_{\infty}}{M (\wt_-) c_1^3} \cA_0 \Big\} := \eta 
    \end{equation}
    Now, given that $\mathcal{A}_0 < \epsilon^2$ (and therefore by alignment, 
    $\cA(t) \leq \epsilon^2$) and $\|\nabla \u_0\|_{\infty} < \epsilon$,
    let $[0,T)$ be the maximal interval of existence and let $[0, t^*)$ be the the maximal time interval on the interval existence
    for which $\|\nabla \u\|_{\infty} < 2 \epsilon$. 
    Let $a = \eta  M \|\phi\|_{\infty} + \wt_+ M \|\nabla \phi\|_{\infty}$ and $b = c_1 \wt_- M$. 
    Then 
    \begin{equation*}
        \frac{d}{dt} \|\partial \u\|_{\infty} 
            \leq a \epsilon^2 - (b - 2\epsilon) \|\partial \u \|_{\infty}, \hspace{5mm} t < t^*
    \end{equation*}
    Integrating, we obtain 
    \begin{equation*}
        \|\partial \u\|_{\infty} \leq  \|\partial \u_0\|_{\infty} + \frac{a \epsilon^2}{b-2\epsilon}
            \leq \epsilon + \frac{a \epsilon^2}{b-2\epsilon}
    \end{equation*}
    Fix $0 < \gamma < 1$. Provided $\epsilon < b(2\gamma - 1)/(4\gamma - 2 + a)$, $\|\partial \u\|_{\infty} < 2 \gamma \epsilon < 2 \epsilon$ for all $t < t^*$. 
    Thus, for small enough $\epsilon$, $t^* = T$ and by the continuation criterion, the solution can be continued beyond $T$, contradicting that it is the maximal time of existence. 
    Thus $T = \infty$.  
    This argument holds for all $\gamma < 1$ so Theorem \ref{small_data} follows. 
\end{proof}

\section{Global Well-posedness in 1D}
\label{GWP_1D}
In this section, 
we prove Theorem \ref{thm:UniGWP} in the 1D case first in order to illustrate the core of the argument before proceeding to the multi-D case, which is proved in Section \ref{UniGWP}. 
We will establish a threshold condition for global well-posedness in 1D 
in a similar manner to Carrillo, Choi, Tadmor and Tan in \cite{CCTT2016}.
The precise statement is as follows. 
\begin{theorem}
    \label{t:GWP_1D}
    Assume (A1)-(A5) and that the initial density is non-vacuous, i.e. $\rho_0 > c > 0$. 
    \begin{enumerate}[leftmargin=0.8cm, label = (\roman*)]
        \item If $e_0 = \partial_x u_0 + \wt_0 (\rho_0)_{\phi} \geq 0$, then $e$ remains positive for all $t>0$ and there's a unique solution 
        $(\rho, \wt, u) \in C_w([0,T]; (H^k \cap L^1_+) \times H^l \times H^m)$ to \eqref{EAS_WM}
        existing globally in time and satisfying the initial data. 
        \item If $e_0 = \partial_x u_0 + \wt_0 (\rho_0)_{\phi} < 0$, then $e$ approaches $-\infty$ in finite time and there is finite time blow-up 
        of the solution.
    \end{enumerate}
\end{theorem}

\begin{proof}
    By the continuation criterion \eqref{cont_criterion}, it suffices to control $\|\partial_x u\|_{\infty}$ and $\inf \rho$. 
    By design, the entropy, $e = \partial_x u + \wt \rho_{\phi}$, in 1D is conserved. We have 
    \begin{equation*}
        \partial_t e + \partial_x(u e) = 0 
    \end{equation*}
    Written along characteristics, we have the ODE 
    \begin{equation*}
        \dot{e} = e(\wt \rho_{\phi} -e)
    \end{equation*}
    Provided $0 < \wt_0 < \infty$, the logistic ODE on $e$ implies 
    \begin{enumerate}[leftmargin=0.8cm, label = (\roman*)]
        \item if $e_0 \geq 0$ then $e(t) > 0$ and it remains bounded. 
        \item if $e_0 < 0$ then $\dot{e} \leq -e^2$ so $e$ blows up. 
    \end{enumerate}
    Therefore we have the threshold condition: if $e_0 < 0$, the solution blows up; 
    but if $e_0 \geq 0$, then $\partial_x u$ remains bounded by some constant $C > 0$. 
    Writing the $\rho$-equation along characteristics, we get 
    \begin{equation*}
        \dot{\rho} = -\partial_x u \rho
    \end{equation*}
    which implies that, along characteristics, 
    \begin{equation*}
        \rho \geq \rho_0 e^{-\int_0^t \|\partial_x u\|_{\infty} ds} = \rho_0 e^{-Ct}
    \end{equation*}
    Due to the assumed non-vacuous initial density, $\rho_0 > c > 0$, this is enough to conclude that $1/\rho_{\phi}$ is bounded above and 
    hence there is global well-posedness via the continuation criterion \eqref{cont_criterion}.  However, the lower bound 
    can be improved to be of order $1/(1+t)$. Let us include the argument for the sake of 
    optimality. It may be relevant to future flocking results. 
    Observe that $e$ and $\rho$ satisfy the same transport equation and a result, 
    the quantity $\frac{e}{\rho}$ is transported. 
    \begin{equation*} 
       \partial_t \big( \frac{e}{\rho} \big) + u \partial_x \big(\frac{e}{\rho} \big) = 0
    \end{equation*}
    In particular, letting $C = \frac{e_0}{\rho_0}$, we have $e \leq C \rho$. 
    Substituting $-\partial_x u = w\rho_{\phi} - e$ in the characteristic equation for $\rho$, we get 
    \begin{equation*}
        \dot{\rho} = (\wt \rho_{\phi} - e) \rho \geq -C \rho^2 
    \end{equation*}
    Consequently, $\dot{(1/\rho)} \leq C$ and integrating we obtain 
    $1/\rho \leq 1/\rho_0 + C t$ along characteristics. Again, due to the non-vacuous initial density, we 
    obtain an upper bound on $1/\rho_{\phi}$ with linear in time growth. 
    In particular, $1/\rho_{\phi}$ is bounded on any finite time interval 
    and, by the continuation criterion, any local solution for which $e_0 \geq 0$ can be extended to any time interval.
\end{proof}

\section{Limiting Density Profile}
\label{sec:limiting_density_profile}
With the conditions for alignment and global well-posedness in 1D at hand, two natural questions arise.
Is there a limiting density distribution for the flock?  If so, what does the limiting density profile look like?  
In the 1D Cucker-Smale case with a heavy tail kernel and $e_0 > 0$, 
the former is answered in the affirmative in \cite{Sbook}; 
the latter is answered partially by establishing an estimate on its 
deviation from the uniform distribution in \cite{LS-entropy}.  We extend these results to the $W$-model. 
In particular, we prove Theorem \ref{thm:strong_flocking_intro} and Theorem \ref{thm:entropy_estimate_intro}.

\subsection{Strong Flocking in 1D} 
\label{sec:strong_flocking}
The solution flocks strongly if there is alignment of the velocities as well as convergence of the density $\rho$ to a limiting distribution $\rho_{\infty}$. 
This can be established in 1D provided $e_0 > 0$ and there is exponential alignment of the velocities, which according to Theorem \ref{thm:linf_alignment}, necessitates 
the bound from below $\phi \geq c_1$ (in particular, $\rho_{\phi} \geq c_1 M$). The strict positivity of $e_0$ guarantees 
dissipation in the $\partial_x u$ equation, which is crucial for strong flocking.  Indeed,
let $u_{\infty}$ denote the limiting velocity.  If there were a limiting density profile, it must be that the time derivative of the density 
in the moving frame with coordinates $x' = x - u_{\infty}t$, $t' = t$ is approaching zero sufficiently fast. 
Examining the equation for the density in the moving frame, 
\begin{equation}
    \label{eq:rho_moving_frame}
    \partial_{t'} \rho + (u - u_{\infty}) \partial_{x'} \rho + (\partial_{x'} u) \rho = 0
\end{equation}
we see that boundedness of $\rho$ and $\partial_x \rho$ along with sufficiently fast decay of $\partial_x u$ is 
sufficient for strong flocking.  
\begin{remark}
    In the case of small data, Theorem \ref{small_data}, the smallness of the initial variation of $u$ led 
    to a dissipative term in the equation for $\partial u$.  Here, we replace the small data assumption with $e_0 > 0$, 
    which, in 1D, also leads to a dissipative term in the equation for $\partial u$.
\end{remark}
\begin{proof}[Proof of Theorem \ref{thm:strong_flocking_intro}]
    First, note that if $e_0 \geq c_2 > 0$, then it remains bounded from below.  
    Indeed, along characteristics, $\dot{e} = e(\wt \rho_{\phi} - e)$ is non-negative 
    whenever $0 \leq e \leq \wt \rho_{\phi}$.  Therefore, $e \geq \min\{ c_2, \wt \rho_{\phi} \} := c$.  
    Let $E(t)$ denote a generic exponentially decaying 
    quantity, which may vary from line to line. From the equation for $\partial_x u$ and $e > c > 0$, we have 
    \begin{equation*}
        \frac{d}{dt} \|\partial_x u\| \leq \big( \wt_+ + \| \partial_x \wt \|_{\infty} \big) E(t) - c \|\partial_x u\|_{\infty} 
    \end{equation*}
    The exponential decay of $u_F$ (given that $\phi$ is bounded below away from zero) was shown in estimate\eqref{ineq:decay_u_F}. 
    As a result, $\|\partial_x \wt\|$ is bounded and $\|\partial_x u\|_{\infty}$ is exponentially decaying. 
    
    Turning to the second derivative, the equation for $\partial^2_x u$ is given by 
    \begin{align*}
        \partial_t (\partial_x^2 u) + u \cdot \partial^3_x u
        &= \partial_x^2 \wt ( (u\rho)_{\phi} - u \rho_{\phi} )
            + 2 \partial_x \wt ( (u\rho)_{\phi'} - u \rho_{\phi'} )  \\
        &+ \wt ( (u\rho)_{\phi''} - u \rho_{\phi''} ) 
            - 2\partial_x(\wt\rho_{\phi})) \partial_x u
            - 2 (\partial_x u) \partial_x^2 u
            - e \partial^2_x u
    \end{align*}
    To control the first term, we will show that $\|\partial^2_x \wt\|_{\infty}$ is bounded. 
    The equation for $\partial^2_x \wt$ is given by 
    \begin{equation*}
        (\partial_t + u_F \partial_x)(\partial^2_x \wt) = - 2(\partial_x u_F) \partial^2_x \wt + (\partial_x^2 u_F) \partial_x \wt 
    \end{equation*}
    We claim that $\partial^2_x u_F$ is exponentially decaying.  Indeed, 
    \begin{align*}
        \partial^2_x u_F 
            &= \frac{(u \rho)_{\phi''}} {\rho_{\phi}} - \frac{\rho_{\phi''} (u\rho)_{\phi}} {\rho_{\phi}^2}
            - 2 \frac{\rho_{\phi'} (u\rho)_{\phi'}} {\rho_{\phi}^2} + 2\frac{\rho_{\phi'}^2 (u\rho)_{\phi}} {\rho_{\phi}^3} \\
            &:= A_1 - A_2 - B_1 + B_2
    \end{align*}
    By the exponential alignment \eqref{thm:linf_alignment}, we have 
    \begin{equation*}
        A_1 - A_2 
            = \frac{ \rho_{\phi} ((u \rho)_{\phi''} - u \rho_{\phi''}) - \rho_{\phi''} ((u\rho)_{\phi} - u\rho_{\phi})} {\rho_{\phi}^2}
            \leq E(t) \\
    \end{equation*}
    Similarly, 
    \begin{equation*}
        B_2 - B_1
            = 2 \frac{\rho_{\phi'}^2 ((u\rho)_{\phi} - u\rho_{\phi}) - \rho_{\phi} \rho_{\phi'} ((u\rho)_{\phi'} - u\rho_{\phi'}) } {\rho_{\phi}^3}
            \leq E(t)
    \end{equation*}
    In total, we gather that 
    \begin{equation*}
        \frac{d}{dt} \| \partial^2_x \wt \|_{\infty} \leq E(t) + E(t) \|\partial^2_x w\|_{\infty}
    \end{equation*}
    In particular, $\|\partial^2_x \wt\|_{\infty}$ is bounded.  Returning to $\partial_x^2 u$, we obtain 
    \begin{equation*}
        \frac{d}{dt} \| \partial_x^2 u \|_{\infty} \leq E(t) + (E(t) - c) \|\partial^2_x u\|_{\infty} 
    \end{equation*}
    Thus $\|\partial^2_x u\|_{\infty}$ is exponentially decaying. 
    With exponential decay of $\|\partial_x u\|_{\infty}$ and $\|\partial^2_x u\|_{\infty}$ in hand, 
    boundedness of $\rho$ and $\partial_x \rho$ follows. The former follows from the 
    continuity equation and latter from 
    \begin{equation*}
        \partial_t (\partial_x \rho) = -u \partial^2_x \rho - 2 (\partial_x u) (\partial_x \rho) - (\partial^2_x u) \rho 
    \end{equation*}
    so that 
    \begin{equation*}
        \partial_t \|\partial_x \rho\|_{\infty} \leq E(t) + E(t) \|\partial_x \rho\|_{\infty}
    \end{equation*}
    Integrating gives a uniform bound on $\|\partial_x \rho\|_{\infty}$. 
    Now, from the density equation in the moving frame \eqref{eq:rho_moving_frame} along 
    with exponential alignment, we have 
    \begin{equation*}
        \partial_t \rho = E(t) 
    \end{equation*}
    In particular, $\rho(t, x)$ is Cauchy in time, uniformly in $x$, so there exists a limiting function 
    $\rho_{\infty}(x)$.  Moreover, the exponential decay of $\partial_t \rho$ implies exponential convergence of $\rho(t,x)$ to $\rho_{\infty}(x)$ in $L^{\infty}$.  
\end{proof}

\subsection{Relative Entropy Estimate and Distribution of Limiting Flock in 1D}
\label{sec:entropy}
In this section, we prove Theorem \ref{thm:entropy_estimate_intro}.
We aim to estimate the $L^1$ distance of the limiting flock to the uniform distribution, $\bar{\rho} = M/2\pi$:
\begin{equation*}
    \limsup_{t \to \infty} \|\rho(t) - \bar{\rho} \|_{L^1} 
        \leq  \Big( \tilde{Q} +  \|\phi\|_{\infty} (\wt_+ - \wt_-) \Big) \frac{M \wt_+ \|\phi\|_{\infty}}{c(\wt_+ \|\phi\|_{L^1} - \tilde{Q})} 
\end{equation*}
Recall that $\tilde{e} = \partial_x u + \cL_{\phi} \rho$, where $\cL_{\phi} \rho = \wt(x) \int_{\T} (\rho(y) - \rho(x)) \phi(x-y) \dy$ 
and $\tilde{q} = \frac{\tilde{e}}{\rho}$; and it is assumed that $\tilde{q} \leq \tilde{Q}$ for some constant $\tilde{Q}$. 
We remark on the conditions for such a constant $\tilde{Q}$ to exist. 
\begin{remark}
    \label{rmk:bd_on_q_tilde}
    The boundedness of $\|\tilde{q}\|_{\infty} \leq \tilde{Q} < \wt_+ \|\phi\|_{L^1}$ for some constant $\tilde{Q}$ is 
    satisfied when $e = \partial_x u + \wt \rho_{\phi}$, the kernel $\phi$, and the weight $\wt$ are bounded away from zero, i.e. $\phi \geq c_1 > 0$, $e_0 \geq c_2 > 0$, 
    $\wt \geq \wt_- > 0$ and there is a smallness assumption on the derivative of the initial weight and/or the initial variation of the velocity. Indeed, 
    if $e_0 > 0$ then it remains bounded away from zero for all time. 
    Since $\tilde{q}$ satisfies $\partial_t \tilde{q} + u \partial_x \tilde{q} = \partial_x \wt (u - u_F)$
    and the kernel is bounded away from zero, 
    we have exponential alignment and, as a result, $\|\partial_x \wt\|_{\infty}$ remains bounded 
    by $\|\partial_x \wt_0\|_{\infty} \exp \Big\{ \frac{2 \|\phi\|_{\infty} \|\phi'\|_{\infty}}{M (\wt_-) c_1^3} \cA_0 \Big\}$, 
    where $\cA(t) = \max_{(x,y) \in \T^2} |u(t,x) - u(t,y)|$.  This was shown in \eqref{eqn:bd_on_w_x}. 
    As a result,
    \begin{align*}
        \|\tilde{q}\|_{\infty} 
            &\leq \|\tilde{q}_0\|_{\infty}  +  \| \partial_x \wt \|_{\infty}  \frac{\cA_0}{ M c_1 \wt_-} \\
            &\leq \|\tilde{q}_0\|_{\infty}  +   \partial_x \|\wt_0\|_{\infty} \exp \Big\{ \frac{2 \|\phi\|_{\infty} \|\phi'\|_{\infty}}{M (\wt_-) c_1^3} \cA_0 \Big\} \frac{\cA_0}{ M c_1 \wt_-}
    \end{align*} 
    We see that small values of $\|\partial_x \wt_0\|_{\infty}$ or $\cA_0$ or $\frac{1}{\wt_-}$ are sufficient 
    to achieve $\|\tilde{q}\|_{\infty} < \wt_+ \|\phi\|_{L^1}$.
\end{remark}
The proof relies on an estimate of the \textit{relative} entropy $\cH = \int_{\T} \rho \log \frac{\rho}{\bar{\rho}}$ in order to achieve the desired estimate. 
This is to be distinguished from the entropy $e$. 
\begin{proof}[Proof of Theorem \ref{thm:entropy_estimate_intro}]
    By the Csisz\'{a}r-Kullback inequality Lemma \ref{lma:ck_inequality}, to control $\|\rho - \bar{\rho}\|_1$, it suffices to control $\cH = \int_{\T} \rho \log \frac{\rho}{\bar{\rho}} \dx$, 
    We have 
    \begin{equation*}
        \partial_t (\rho \log \rho) 
            = -[u (\rho \log \rho)]' - \rho u' 
            = -[u (\rho \log \rho)]' - \rho \tilde{e} - \rho \cL\rho
    \end{equation*}
    So that, 
    \begin{align*}
        \frac{d \cH}{dt} 
            &= \frac{d}{dt} \int_{\T} \rho \log \rho \dx  \\
            &= \int_{\T} (\rho - \bar{\rho}) \tilde{e} \dx - \int_{\T^2} \rho(x) \wt(x) (\rho(y) - \rho(x)) \phi(x-y) \dy \dx \\
            &= -\int_{\T} (\rho - \bar{\rho}) \tilde{e} \dx - \bar{\rho} \int_{\T} \tilde{e} \dx + \int_{\T^2} \rho(x) \wt(x) (\rho(y) - \rho(x)) \phi(x-y) \dy \dx \\
            &:= I_1 + I_2 + I_3
    \end{align*}
    Symmetrizing the $I_3$ term, we obtain 
    \begin{align*}
        I_3 &= -\frac{1}{2} \int_{\T^2} \wt(x) |\rho(x) - \rho(y)|^2 \phi(x-y) \dy \dx 
                - \frac{1}{2} \int_{\T^2} \rho(y) (\wt(x) - \wt(y)) (\rho(x) - \rho(y)) \phi(x-y) \dy \dx \\
            &\leq -\frac{1}{2} \wt_- \int_{|x-y| \leq r_0} |\rho(y) - \rho(x)|^2 \dy \dx 
                + \frac{1}{2} \|\rho\|_{\infty} \|\phi\|_{\infty} (\wt_+ - \wt_-) \int_{\T^2} |\rho(x) - \bar{\rho} - (\rho(y) - \bar{\rho})| \dy \dx \\
            &\leq -\frac{1}{2} \wt_- \int_{|x-y| \leq r_0} |\rho(y) - \rho(x)|^2 \dy \dx 
                + 2\pi \|\rho\|_{\infty} \|\phi\|_{\infty} (\wt_+ - \wt_-) \|\rho - \bar{\rho}\|_{L^1}  \\
    \end{align*}
    By Lemma 2.1 of \cite{LS-entropy}, we have, for some positive constant $c := c(r_0)$, 
    \begin{equation*}
        \leq -c \wt_- \|\rho - \bar{\rho}\|_{L^2}^2
                + 2\pi \|\rho\|_{\infty} \|\phi\|_{\infty} (\wt_+ - \wt_-) \|\rho - \bar{\rho}\|_{L^1}  
    \end{equation*}
    In the remainder, $c$ may change from line to line, but it will remain solely dependent on $r_0$. Symmetrizing $I_2$, we have
    \begin{align*}
        I_2 &= \bar{\rho}  \int_{\T^2} \wt(x) (\rho(y) - \rho(x)) \phi(x-y) \dy \dx \\
            &= \frac{1}{2} \bar{\rho}  \int_{\T^2} (\wt(y) - \wt(x)) (\rho(y) - \rho(x)) \phi(x-y) \dy \dx \\
    \end{align*} 
    Using $\bar{\rho} \leq \|\rho\|_{\infty}$, we obtain the same estimate as the non-dissipative term in $I_3$.
    \begin{equation*}
        |I_2| \leq 2\pi \|\rho\|_{\infty} (\wt_+ - \wt_-) \|\phi\|_{\infty} \|\rho - \bar{\rho}\|_{L^1} 
    \end{equation*}
    For $I_1$, we have 
    \begin{align*}
        |I_1| &= \big| \int_{\T} (\rho - \bar{\rho}) \rho \tilde{q} \dx \big| \\
              &= \|\rho\|_{\infty} \|\tilde{q}\|_{\infty} \|\rho - \bar{\rho}\|_{L^1}
    \end{align*}
    Combining these estimates with the the Csisz\'{a}r-Kullback inequality, we obtain 
    \begin{align*}
        \frac{d \cH}{dt} 
            &\leq \Big( \|\rho(t)\|_{\infty} \|\tilde{q}(t)\|_{\infty} + 4\pi \|\rho(t)\|_{\infty} \|\phi\|_{\infty} (\wt_+ - \wt_-) \Big) \|\rho(t) - \bar{\rho}\|_{L^1} - c \wt_- \|\rho(t) - \bar{\rho}\|_{L^2}^2 \\
            &\leq \Big( \|\rho(t)\|_{\infty} \|\tilde{q}(t)\|_{\infty} + 4\pi \|\rho(t)\|_{\infty} \|\phi\|_{\infty} (\wt_+ - \wt_-) \Big) \sqrt{4 \pi \bar{\rho} \cH(t)} - c \wt_- \bar{\rho} \cH(t)  
    \end{align*}
    Letting $Y = \sqrt{\cH}$, we obtain via Gronwall,   
    \begin{align*}
        Y(t) &\leq Y_0 e^{-c \wt_- \bar{\rho}t}  
        + \sqrt{\pi \bar{\rho}} \int_0^t  \|\rho(s)\|_{\infty} \|\tilde{q}(s)\|_{\infty}  e^{-c \bar{\rho}(t-s)} ds  \\       
        &+ 4\pi \sqrt{\pi \bar{\rho}} \|\phi\|_{\infty} (\wt_+ - \wt_-) \int_0^t  \|\rho(s)\|_{\infty}  e^{-c \bar{\rho}(t-s)} ds  \\       
    \end{align*}
    To relate it back to $\|\rho - \bar{\rho}\|_{L^1}$, we multiply both sides of the inequality by $\sqrt{4\pi \bar{\rho}}$ 
    and apply the Csisz\'{a}r-Kullback inequality again.  Combining this with the following elementary fact: 
    for a bounded function $a(t)$ and a constant $b$, $\limsup_{t \to \infty} \int_0^t a(s) e^{-b(t-s)} ds \leq \frac{1}{b} \limsup_{t \to \infty} a(t)$,
    we obtain (since $\rho$ is bounded by Lemma \ref{lma:bd_on_rho})
    \begin{align*}
        \limsup_{t \to \infty} \|\rho(t) - \bar{\rho}\|_{L^1} 
            &\leq \frac{1}{c} \limsup_{t \to \infty} \|\rho(t)\|_{\infty} \|\tilde{q}(t)\|_{\infty} \\
            &+ \frac{1}{c} \|\phi\|_{\infty} (\wt_+ - \wt_-) \limsup_{t \to \infty} \|\rho(t)\|_{\infty} \\
    \end{align*}
    Applying Lemma \ref{lma:bd_on_rho} to bound $\limsup_{t \to \infty} \|\rho(t)\|_{\infty}$ gives the result.  
\end{proof}

\begin{lemma}
    \label{lma:bd_on_rho}
    If there exists a constant $\tilde{Q}$ such that $\|\tilde{q}\|_{\infty} \leq \tilde{Q} < \wt_+ \|\phi\|_{L^1}$, then 
    \begin{equation*}
        \limsup_{t \to \infty} \|\rho(t)\|_{\infty} \leq \frac{M \wt_+ \|\phi\|_{\infty}}{\wt_+ \|\phi\|_{L^1} - \|\tilde{q}\|_{\infty}}
    \end{equation*}
\end{lemma}
\begin{proof}
    Let $x_+$ denote the maximizer of $\rho(t)$. That is, $\rho_+(t) = \rho(t, x_+)$. From the continuity equation, we have 
    \begin{align*}
        \frac{d}{dt} \rho_+(t) 
            &= -\rho_+(t) \partial_x u(t, x_+) = -\rho_+(t) (\tilde{e} - \cL_{\phi} \rho) \\
            &= -\rho_+(t)^2 \tilde{q}(t, x_+) + \rho_+(t) \wt(t, x_+) \int_{\T} \phi(x_+ - y) (\rho(t,y) - \rho_+(t)) \dy \\
            &\leq  (\tilde{Q} - \wt_+ \|\phi\|_{L^1}) \rho_+(t)^2 + M \wt_+ \|\phi\|_{\infty} \rho_+(t) \\
            &=  (\wt_+ \|\phi\|_{L^1} - \tilde{Q}) \rho_+(t) \bigg[ \frac{M \wt_+ \|\phi\|_{\infty}}{\wt_+ \|\phi\|_{L^1} - \tilde{Q}} - \rho_+(t) \bigg] \\
    \end{align*}
    Observe that if $\dot{X}(t) \leq A X(t) [B - X(t)]$ where $A,B > 0$ are constants and $X(t) > 0$, 
    then 
    \begin{equation*}
        X(t) \leq \frac{BX(0)} {X(0) + (B - X(0)) \exp(-ABt)}
    \end{equation*}
    Applying this differential inequality and taking $t \to \infty$ gives the result. 
\end{proof}

\section{Unidirectional Flocks}
\label{UniGWP}
In this section, we prove Theorem \ref{thm:UniGWP} in full.  As in the 1D case, Theorem \ref{t:GWP_1D},
the $e$-quantity is used to control the gradient of the velocity.  The difference here is that the full gradient 
needs to be controlled.  Let us recall the definition of unidirectional flocks \eqref{e:uniintro}. 
A flock is unidirectional if it has the form 
\begin{equation*}
    \u(x,t) = u(x,t) \bd, \hspace{8mm} \bd \in \mathbb{S}^{n-1}, \hspace{2mm} \u: \T^n \times \R^+ \to \R
\end{equation*}
for all time $t$.  The precise statement to be proved is as follows. 
\begin{theorem}
    \label{t:uni}
    Assume (A1)-(A5), the initial density is non-vacuous, i.e. $\rho_0 > c > 0$, and that $\u_0$ is unidirectional in the direction $\bd$. 
    \begin{enumerate}[leftmargin=0.8cm, label = (\roman*)]
        \item If $e_0 = \nabla u_0 \cdot \bd + \wt_0 (\rho_0)_{\phi} \geq 0$, then there's a unique solution 
          $(\rho, \wt, \u) \in C_w([0,T]; (H^k \cap L^1_+) \times H^l \times H^m)$ to \eqref{EAS_WM}
          existing globally in time and satisfying the initial data. 
        \item If $e_0 = \nabla u_0 \cdot \bd + \wt_0 (\rho_0)_{\phi} < 0$, then there is finite time blow-up 
        of the solution.
    \end{enumerate}
\end{theorem}

\begin{proof}
    By the continuation criterion \eqref{cont_criterion}, it suffices to control $\|\nabla \u\|_{\infty}$ and $\inf \rho$. 
    To do so, we will write the 1D system for the component nonzero component $u$ in order to exploit the $e$-quantity as in 
    Theorem \ref{t:GWP_1D}. First, note that the maximum principle applied to each direction implies that the solution $\u$ remains unidirectional for all time; 
    and by rotation invariance of \eqref{EAS_WM}, we can assume WLOG that $\bd = x_1$. 
    The velocity then takes the form $\u = (u(x,t), 0, \dots, 0)$ for all time and the system \eqref{EAS_WM} can be written 
    \begin{equation*}
        \label{SM_uni}
        \begin{cases}
            \partial_t \rho + \partial_1 (u\rho) = 0 \\
            \partial_t \wt + [u]_{\rho} \partial_1 \wt = 0\\
            \partial_t u + u \partial_1 u = \wt ((u\rho)_{\phi} - u\rho_{\phi})
        \end{cases}
    \end{equation*}
    The entropy equation is given by
    $$
        e = \partial_1 u + \wt \rho_{\phi}, \hspace{5mm} \partial_t e + \partial_1(ue) = 0 
    $$
    Written along characteristics, we recover the same ODE from the 1D case.  
    \begin{equation*}
        \dot{e} = e(\wt \rho_{\phi}-e)
    \end{equation*}
    As a result, we obtain the same threshold condition. 
    \begin{enumerate}[leftmargin=0.8cm, label = (\roman*)]
        \item if $e_0 \geq 0$ then $e(t) > 0$ bounded. 
        \item if $e_0 < 0$ then $\dot{e} \leq -e^2$ so $e$ blows up. 
    \end{enumerate}
    Considering the case $e_0 \geq 0$, the bound $1/\rho \leq 1/\rho_0 + C_1 t$ along characteristics follows a similar argument to the 1D case. 
    Turning to $\nabla u$, we write the equation for a generic partial derivative $\partial u$. 
    \begin{align*}
        (\partial_t + u \partial_1) \partial u 
            &= -(\partial_1 u)(\partial u) + \partial \big( \wt((u\rho)_{\phi} - u \rho_{\phi}) \big) \\
            &= -e (\partial u) + \partial \wt ((u\rho)_{\phi} - u \rho_{\phi}) + \wt ((u\rho)_{\phi'} - u \rho_{\phi'}) \\
    \end{align*}
    Multiplying by $\partial u$ and taking the supremum over the support of $\rho$ and using that $e \geq 0$, 
    we get  
    \begin{align*}
        \frac{d}{dt} \|\partial u\|_{L^{\infty}}  
            &\leq \|\partial \wt\|_{\infty} \|(u\rho)_{\phi} + u \rho_{\phi}\|_{\infty} + \|\wt ((u\rho)_{\phi'} - u \rho_{\phi'}) \|_{\infty} \\
            &\lesssim \|u\|_{\infty} (\wt_+ + \|\partial \wt\|_{\infty})
    \end{align*}
    It remains to bound $\|\partial \wt\|_{\infty}$. We have 
    \begin{equation*}
        (\partial_t + u_F \partial_1) \partial \wt + (\partial u_F)  (\partial_1 \wt) = 0 
    \end{equation*}
    In the case that $\partial = \partial_1$, solving along characteristics, we have 
    \begin{align*}
        |\partial_1 \wt|
            &= \big| (\partial_1 \wt_0) \exp \Big\{ - \int_0^t \partial_1 u_F ds \Big\} \big| \\
            &\leq |\partial_1 \wt_0|
                \exp \Big\{ \int_0^t C \Big\| \frac{1}{\rho_{\phi}} \Big\|_{\infty}^2 \|u\|_{\infty} ds \Big\} \\
            &\leq |\partial_1 \wt_0|
                \exp \Big\{ \int_0^t C t^2 \|u\|_{\infty} ds \Big\}
    \end{align*}
    For an arbitrary partial derivative along characteristics, we obtain 
    \begin{align*}
        |\partial \wt|
            &= \big| \partial \wt_0 \int_0^t (\partial u_F) (\partial_1 \wt) ds \big|  \\
            &\leq |\partial \wt_0| \int_0^t Ct^2 \|u\|_{\infty}  |\partial_1 \wt_0|
            \exp \Big\{ \int_0^s C t^2 \|u\|_{\infty} dr \Big\} ds 
    \end{align*}
    which is bounded for finite times. Moreover $\|\partial u\|_{\infty}$ is bounded for finite times for an arbitrary partial derivative $\partial$.
    In other words, the full gradient $\|\nabla \u\|_{\infty}$ is bounded for finite times. 
    The continuation criterion implies that a local solution can be extended to any finite time interval.
\end{proof}

\section{Numerical Simulation Plots and Description}
    \label{description_of_numerics} 
    To illustrate that the $\st$-model with Motsch-Tadmor 
    initial data ($\wt_0 = (1/\rho_0)_{\phi}$) possesses similar qualitative features 
    to the Motsch-Tadmor model, we provide numerical solution plots for these two cases 
    and for the Cucker-Smale case for comparison.  
    The general scheme is a linearized Finite Element discretization 
    in space with conforming elements and semi-implicit backward Euler finite differences in time. 
    The stability and error analysis for this numerical scheme are not known.  
    However, we provide evidence of error convergence in Section \ref{appdx:convg_experiment}
    by plotting the error between a known solution and the numerical solution as the mesh parameters go to zero. 
    Before writing the variational problem, let us describe the discretization of the torus and the numerical solution spaces. 
    
    \subsection{Local polynomial vector spaces and variational problem}
    The discretization of the torus is a uniform partition of the interval $[0,1]$ into $M$ pieces of size $h = 1/M$, 
    where the point $0$ is identified with the point $1$. We will refer each subinterval of size $h$ as an element.  
    The numerical solutions will lie in discrete finite element-based vector spaces with local 3rd and 2nd order local polynomial basis functions, 
    which we denote $P_3$ and $P_2$, respectively. 
    We provide details of the construction of the basis functions for $P_3$; the construction of the basis functions for $P_2$ is analogous. 
    There are $3M + 1$ nodes (for $P_2$, there are $2M + 1$ nodes) placed uniformly over the unit interval.  For each node, there will be 
    a corresponding basis function (which is a piecewise 3rd order polynomial) with support only in nearby elements whose value is equal to $1$ at the given node and $0$ at nearby nodes.  
    To describe the basis functions associated to each node, it is convenient to describe the basis functions whose support intersects a given element. 
    To that end, suppose the given element is $[0,1]$ (the basis functions with support in element $[0,1]$ can easily be adapted 
    to any element by shifting and scaling, which is described later).  The four nodes, placed at positions $0$, $1/3$, $2/3$, and $1$, correspond to 
    four basis functions, which on the interval $[0,1]$, have the form $\psi_k(x) = a_{k0} + a_{k1} x + a_{k2} x^2 + a_{k3} x^3$, $k = 0$ to $3$. 
    The coefficients are chosen so that the $\psi_k(x)$ is equal to $1$ at the node $k/3$ and $0$ at the other three nodes.  In particular, 
    the coefficients are given by the solution to the matrix equation,
    \begin{equation*}
        \begin{bmatrix}
            1 & 0 & 0 & 0 \\
            1 & 1/3 & 1/3^2 & 1/3^3 \\
            1 & 2/3 & (2/3)^2 & (2/3)^3 \\
            1 & 1 & 1 & 1 \\
        \end{bmatrix} 
        \begin{bmatrix}
            a_{00} & a_{01} & a_{02} & a_{03} \\
            a_{10} & a_{11} & a_{12} & a_{13} \\
            a_{20} & a_{21} & a_{22} & a_{23} \\
            a_{30} & a_{31} & a_{32} & a_{33} \\
        \end{bmatrix} 
        = 
        \begin{bmatrix}
            1 & 0 & 0 & 0 \\
            0 & 1 & 0 & 0 \\
            0 & 0 & 1 & 0 \\
            0 & 0 & 0 & 1 \\
        \end{bmatrix}
    \end{equation*} 
    Now let us describe how the entire basis functions are formed from these $\psi_k$.
    Let $v_k^*$, $k = 0$ to $3$, be the basis function associated to node at position $k/3$. 
    We will choose $v_k^*$ so that they are continuous at the boundary of the element 
    (however, there is no continuity of the derivatives of the boundary of the element).
    In particular, the $v_k^*$ are constructed using $\psi_k$ as follows. 
    We simply choose $v_1^* = \psi_1 \mathcal{X}_{[0,1]}$ and $v_2^* = \psi_2 \mathcal{X}_{[0,1]}$. 
    Continuity at the boundary of the element holds by the construction of $\psi_1$, $\psi_2$. 
    The functions $\psi_0$ and $\psi_3$ are equal to $1$ at the boundary of the element.  To retain 
    continuity at the boundary, let $\psi_k'$, $\psi_k''$ denote the basis functions with support
    in the left adjacent element $[-1,0]$ and the right adjacent element $[1,2]$, respectively 
    (that is, $\psi_k'(x) = \psi_k(x+1)$ and $\psi_k''(x) = \psi_k(x-1)$). 
    Then $v_0^* = \psi_3' \mathcal{X}{[-1,0]} + \psi_0 \mathcal{X}{[0,1]}$ and $v_3^* = \psi_3 \mathcal{X}{[0,1]} + \psi_0'' \mathcal{X}{[1,2]}$.  
    A general basis function $v_k$ on the mesh with $M$ elements of size $h$ is obtained by shifting and scaling the $v_k^*$'s. 
    For instance, consider the $i^{th}$ element $[\frac{i}{M}, \frac{i+1}{M}]$, $0 \leq i \leq M-1$, on a mesh of $M$ elements. 
    The four basis functions associated to the nodes on this element are given by $v_k^i = v_k^*(M(x - i/M))$, $k = 0$ to $3$. 
    The trial and test function spaces for $P_3$ are both equal and we denote them by $\{ v_k \}_{k=1}^{3M+1}$, as there are $3M+1$ nodes. 
    Similarly, the trial and test function spaces for $P_2$ are both equal and follow a similar construction.  
    We denote the trial and test functions for $P_2$ by $\{ q_k \}_{k=1}^{2M+1}$.

    Now let $(\rho^n, \wt^n, u^n) \in (P_3, P_3, P_2)$ be the numerical solution at the $n^{th}$ time step. 
    Then for some coefficients $b_k^n$, $(b^n_k)'$, and $c_k^n$ and $(v_k, q_k) \in (P_3, P_2)$. 
    \begin{equation*}
        \rho^n(x) = \sum_{k=0}^{3M + 1} b_k^n v_k, \hspace{5mm}
        \wt^n(x) = \sum_{k=0}^{3M + 1} (b^n_k)' v_k, \hspace{5mm}
        u^n(x) = \sum_{k=0}^{2M + 1} c_k^n q_k  
    \end{equation*}
    Let $V_I, Q_I$ be the interpolant operators on $P_3$ and $P_2$, respectively. 
    Given initial data $(\rho_0, \wt_0, u_0)$, we set $(\rho^0, \wt^0, u^0) = (V_I \rho_0, V_I \wt_0, Q_I u_0)$; 
    and we set $\phi_h = V_I \phi$. 
    To obtain the solutions at the next time step, we solve the following variational problem. 
    For all test functions $(v, q) \in (P_3, P_2)$, 
    \begin{align}
        \label{eqn:numerical_scheme}
        \begin{cases}
            \frac{1}{k}  \lan \rho^{n+1} - \rho^n, v \ran  -  \lan \rho^{n+1} u^n, \frac{d}{dx} v \ran  =  0 \\ 
            \frac{1}{k}  \lan \wt^{n+1} - \wt^n, v \ran    +   \lan (\frac{d}{dx} \wt^{n+1}) \frac{(u^n \rho^n)_{\phi_h}}{\rho^n_{\phi_h}}, v  \ran  = 0 \\ 
            \frac{1}{k}  \lan u^{n+1} - u^n, q \ran  +  \lan u^{n+1} \frac{d}{dx} u^n, q \ran   =    \lan \wt^n (u^n \rho^n)_{\phi_h}, q \ran -  \lan \wt^n u^{n+1} \rho^n_{\phi_h}, q \ran
        \end{cases}
    \end{align}
    For the original Motsch-Tadmor model, the weight is set to $1/(\rho^n)_{\phi}$.  The variational problem is given by 
    \begin{align}
        \label{eqn:numerical_scheme_MT}
        \begin{cases}
            \frac{1}{k}  \lan \rho^{n+1} - \rho^n, v \ran  -  \lan \rho^{n+1} u^n, \frac{d}{dx} v \ran  =  0 \\ 
            \frac{1}{k}  \lan u^{n+1} - u^n, q \ran  +  \lan u^{n+1} \frac{d}{dx} u^n, q \ran   =    \lan \frac{1}{(\rho^n)_{\phi}} (u^n \rho^n)_{\phi_h}, q \ran -  \lan \frac{1}{(\rho^n)_{\phi}} u^{n+1} \rho^n_{\phi_h}, q \ran
        \end{cases}
    \end{align}
    The term $\lan \rho^{n+1} u^n, \frac{d}{dx} v \ran$ in both variational forms is obtained via integration by parts. 
    The spaces $P_3$, $P_2$, the assembly of the variational problem, and the numerical solutions to \eqref{eqn:numerical_scheme} were computed with 
    the aid of the FENICS software library \cite{LoggWells2010} \cite{LoggEtal_10_2012}. 

    \begin{remark}
        The choice $(P_3, P_3, P_2)$ was chosen purely heuristically in order to resemble the inf-sup stability condition for the Stokes equation. However, it has not been proven that these spaces satisfy the inf-sup condition for the system \eqref{eqn:numerical_scheme}.
    \end{remark}

    \subsection{Comparison of Cucker-smale and $\st$-model with Motsch-Tadmor initial data}
    \label{appdx:comparison_CS_WM}
    Numerical solution plots for the solution to 
    \eqref{eqn:numerical_scheme} in the Cucker-Smale case ($\wt_0 = 1$) and for the $\st$-model with Motsch-Tadmor initial data ($\wt_0 = 1/(\rho_0)_{\phi}$) 
    are given in tables \ref{plots:solns_CS} and \ref{plots:solns_WM_MT}.  
    The numerical solution plots for the solution to \eqref{eqn:numerical_scheme_MT}, for the original Motsch-Tadmor model, is given in table \ref{plots:solns_MT}.
    The parameters for all three cases are given in \ref{appdx:parameters}.

    \begin{table}[!h]
        \begin{center}
                \begin{tabular}{cc}
                    \includegraphics[width=0.47\linewidth]{\detokenize{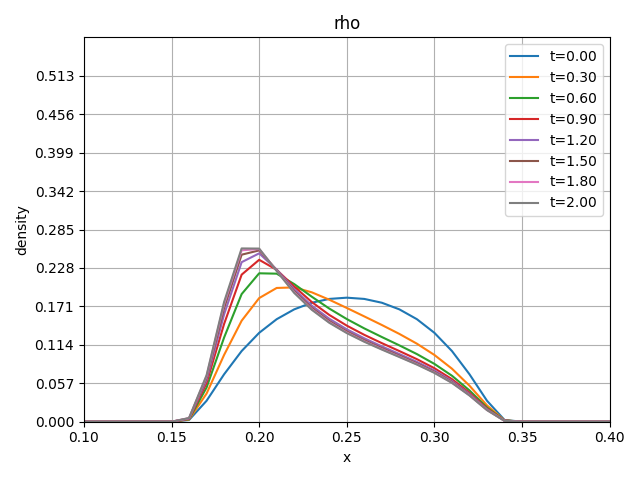}} &
                    \includegraphics[width=0.47\linewidth]{\detokenize{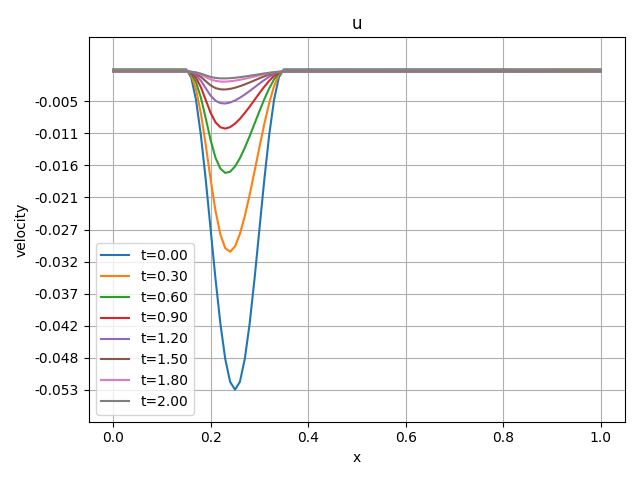}} 
                \end{tabular}
        \end{center} 
        \caption{The computed solution densities zoomed into the small flock and the computed velocities for the Cucker-Smale case.}
        \label{plots:solns_CS}
    \end{table}

    \begin{table}[!h]
        \begin{center}
            \begin{tabular}{cc}
                \includegraphics[width=0.47\linewidth]{\detokenize{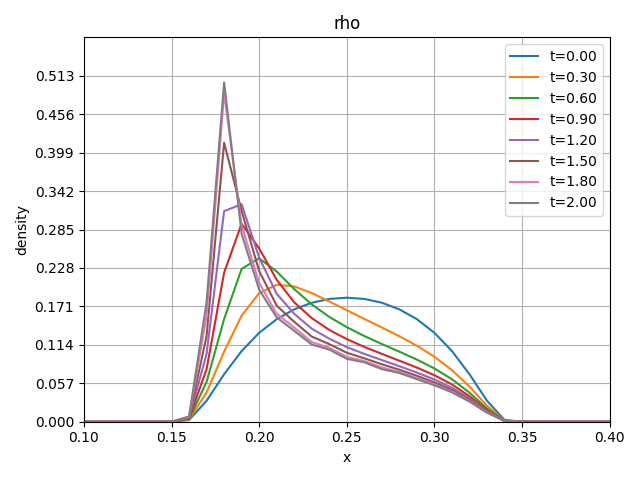}} &
                \includegraphics[width=0.47\linewidth]{\detokenize{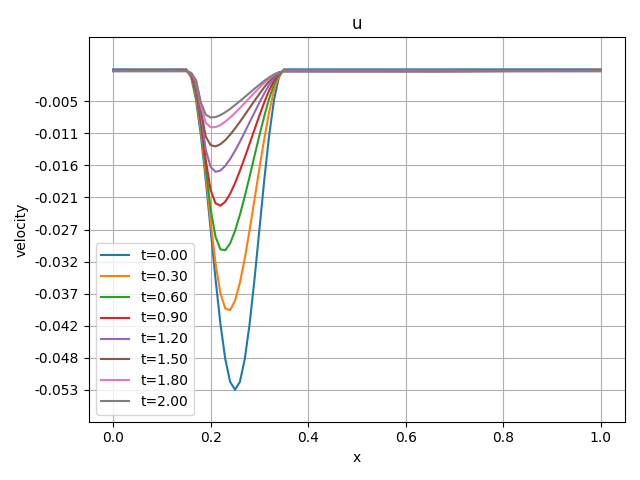}}
            \end{tabular}
        \end{center} 
        \caption{The computed solution densities zoomed into the small flock and the computed velocities for the $\st$-model with Motsch-Tadmor initial data.}
        \label{plots:solns_WM_MT}
    \end{table}
    
    \begin{table}[!h]
        \begin{center}
             \begin{tabular}{cc}
                 \includegraphics[width=0.47\linewidth]{\detokenize{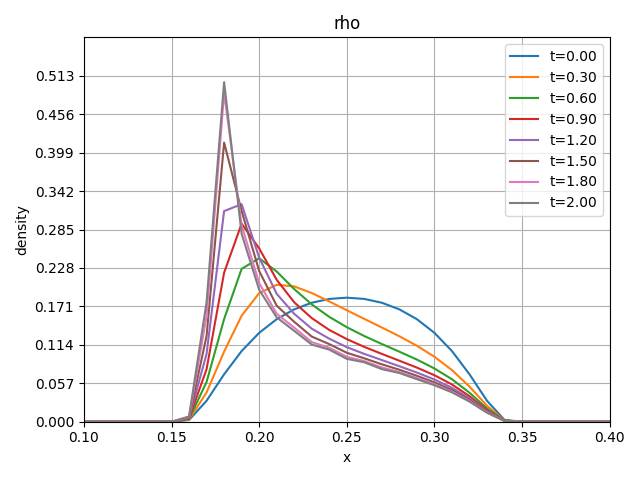}} &
                 \includegraphics[width=0.47\linewidth]{\detokenize{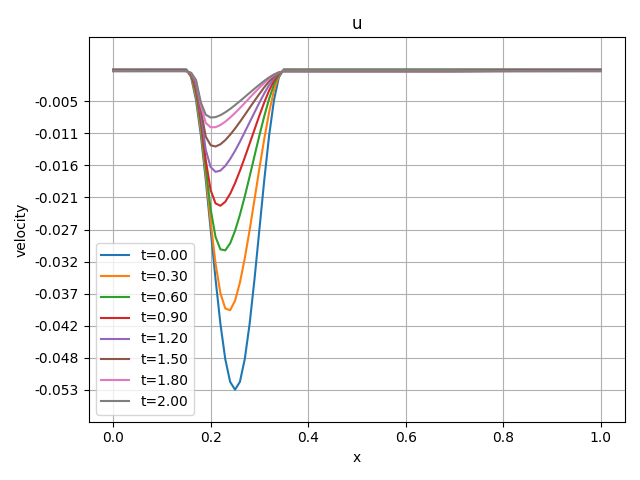}}
             \end{tabular}
         \end{center} 
         \caption{The computed solution densities zoomed into the small flock and the computed velocities for the Motsch-Tadmor case.}
         \label{plots:solns_MT}
     \end{table}

    In the Cucker-Smale case, there is rapid decay of the velocity of the small flock (i.e. rapid alignment to the large flock's velocity) and, 
    as a result, less movement in the density of the small flock.  Conversely, in the case of $\st$-model with Motsch-Tadmor initial data and for the 
    original Motsch-Tadmor model, the velocity decays at a slower rate so there is more movement in the density of small flock. 
    The point we are highlighting here is that the dynamics of the small flock in the Cucker-Smale case gets overpowered 
    by the large flock, while in the case of the $\st$-model with Motsch-Tadmor initial data (and, of course, in the original Motsch-Tadmor model as stipulated in \cite{MT2011}), it does not. 

    \begin{remark}
        The $\st$-model and Motsch-Tadmor cases appear to have identical plots.  This is due to the fact that 
        the large flock remains almost stationary and, as a result, $\rho_{\phi}$ is almost stationary.  
        Even though the differences of the solutions are not perceptible, the Motsch-Tadmor model
        does not have global well-posedness analysis, unlike the $\st$-model as we demonstrated across the paper.
    \end{remark}

    \begin{table}[!h]
        \begin{center}
            \begin{tabular}{cc}
                \includegraphics[width=0.39\linewidth]{\detokenize{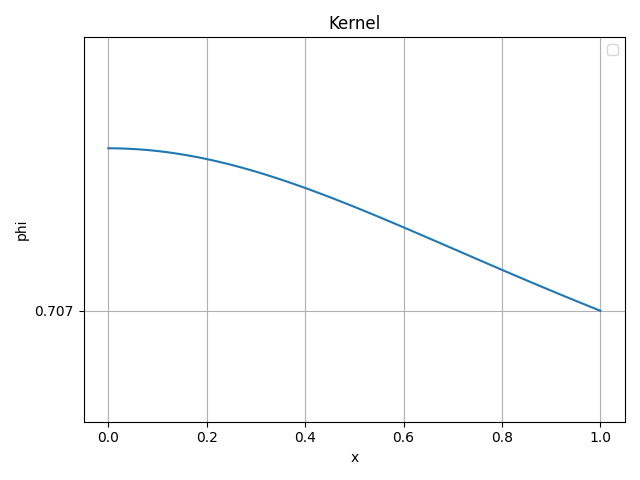}} &
                \includegraphics[width=0.39\linewidth]{\detokenize{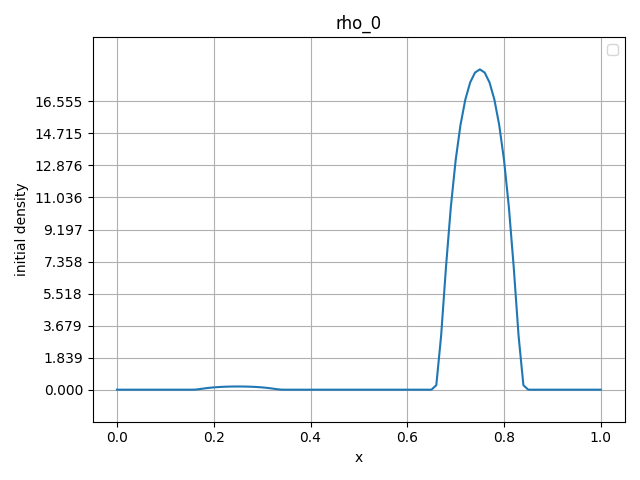}} \\
                \includegraphics[width=0.39\linewidth]{\detokenize{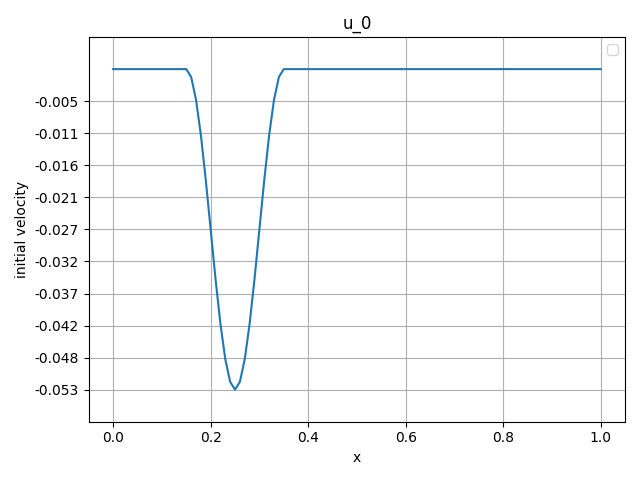}} &
                \includegraphics[width=0.39\linewidth]{\detokenize{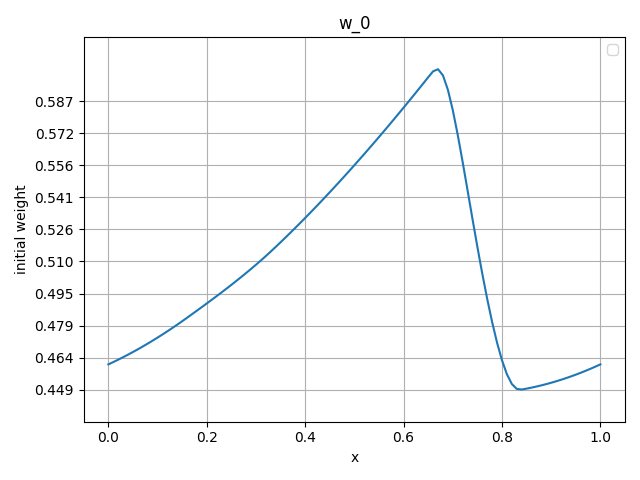}} 
            \end{tabular}
        \end{center} 
        \caption{The kernel $\phi$, initial density $\rho_0$, initial weight for the case of Motsch-Tadmor initial data $\wt_0 = 1/(\rho_0)_{\phi}$ and initial velocity $u_0$.  
        Note that these graphs are at different scales.}
        \label{plots:kernel_and_initial_data}
    \end{table}

    \subsection{Parameters and Initial Data for the Simulation}
    \label{appdx:parameters}
    In the numerical experiments shown described in \ref{appdx:comparison_CS_WM}, 
    the final time $T$ is equal to $2$, the number of time steps is $40$ (i.e. the temporal mesh size $k = 2 / 40$), and the number of mesh elements is $M = 100$ (i.e. $h = 1/100$). 
    The initial density comprises of a small mass flock and a large mass flock.  
    The initial velocity gives the small flock a negative velocity and the large flock a zero velocity.
    The kernel and initial data are shown in table \ref{plots:kernel_and_initial_data} and are given 
    explicitly by 
    \begin{enumerate} 
        \item[] \[ 
            \rho_0(x) = \begin{cases} 
                    \frac{1}{2} \exp \Big\{-\frac{1}{1 - (10*(x-0.25)^2)} \Big\} \hspace{8mm} \text{if } 0.15 < x < 0.35 \\
                    50 \exp \Big\{-\frac{1}{1 - (10*(x-0.75)^2)} \Big\} \hspace{8mm} \text{if } 0.65 < x < 0.85 \\
                    0 \hspace{8mm} \text{otherwise} 
                \end{cases}
           \]
        \item[] \begin{equation*}
            u_0(x) = \begin{cases} 
                    \frac{1}{12 \pi} \cos(10 \pi (x - 0.15)) - \frac{1}{12 \pi} \hspace{8mm} \text{if } 0.15 < x < 0.35 \\
                    0 \hspace{8mm} \text{otherwise}
                \end{cases}
            \end{equation*}
        \item[] \begin{equation*}
            \phi(x) = \frac{1}{(1+x^2)^{1/2}}
        \end{equation*}
    \end{enumerate} 
    The constant of $1/(12\pi)$ was chosen to guarantee that $e_0 = \partial_x u_0 + \wt_0 (\rho_0)_{\phi} > 0$ (so the solution will not blow up, see 
    Section \ref{GWP_1D} for details on this threshold condition). 
    In the Cucker-Smale simulation $\wt_0 = 1$ (and therefore remains $1$ for all time). 
    In the case of Motsch-Tadmor initial data, $\wt_0 = 1/(\rho_0)_{\phi}$.

        \subsection{Numerical Convergence Experiment}
        \label{appdx:convg_experiment}
        The well-posedness and error analysis of the numerical scheme is not analyzed here.  
        Instead, we provide evidence of convergence to a true solution as $k, h \to 0$.  
        Observe that $\rho(t,x) = 1 + \sin(t)$, $\wt(t,x) = \sin(t) + \frac{1}{2\pi} (2 + \sin(2\pi x))$, 
        and $u(t,x) = \sin(t) + \frac{1}{2\pi} \sin(2\pi x)$ is a solution to the $\st$-model system with a forcing given by 
        \begin{align}
            \label{eqn:wm_with_forcing}
            \begin{cases}
                \partial_t \rho + \nabla \cdot (u\rho) = \cos(t) + \sin(t) \cos(2 \pi x) + \cos(2 \pi x) \\
                \partial_t \wt + u_F \cdot \nabla \wt = \cos(t) + \sin(t) \cos(2 \pi x) \\
                \partial_t u + u \cdot \nabla u  = \wt ((u \rho)_{\phi} - u \rho_{\phi}) + 
                \cos(t)  +  \sin(t) \cos(2 \pi x)  +  \frac{1}{4 \pi} \sin(4 \pi x)   +   \big( \sin(t)  +  \frac{1}{\pi}  +  
                \\\frac{1}{2  \pi} \sin(2 \pi x) \big)  \big( \frac{1}{2 \pi}  \sin(2 \pi x) + \frac{1}{2 \pi}  \sin(t)  \sin(2 \pi x) \big) 
            \end{cases}
        \end{align}    
        The corresponding variational problem with a forcing $f = (f_1, f_2, f_3)$ is given by 
        \begin{align}
            \label{eqn:numerical_scheme_with_forcing}
            \begin{cases}
                \frac{1}{k}  \lan \rho^{n+1} - \rho^n, v \ran     -   \lan \rho^{n+1} u^n, \frac{d}{dx} v \ran    =   \lan f_1, v \ran   \\ 
                \frac{1}{k}  \lan \wt^{n+1} - \wt^n, v \ran   +  \lan (\frac{d}{dx} \wt^{n+1}) \frac{(u^n \rho^n)_{\phi_h}}{\rho^n_{\phi_h}}, v  \ran     =   \lan f_2, v \ran  \\ 
                \frac{1}{k}  \lan u^{n+1} - u^n, q \ran   +  \lan u^{n+1} \frac{d}{dx} u^n, q \ran   =   \lan \wt^n (u^n \rho^n)_{\phi_h}, q \ran -  \lan \wt^n u^{n+1} \rho^n_{\phi_h}, q \ran + \lan f_3, q \ran
            \end{cases}
        \end{align}    
        We will provide evidence that the numerical solution to \eqref{eqn:numerical_scheme_with_forcing}, with forcing $f$ equal to the right hand side of \eqref{eqn:wm_with_forcing}, 
        converges to the solution to \eqref{eqn:wm_with_forcing}, which in turn provides evidence that the original numerical scheme \eqref{eqn:numerical_scheme} is a 
        (conditionally) stable and convergent scheme.  Convergence is measured in the $H^1$ norm.  
        
        To distinguish the numerical solutions, let $(\rho_h, \wt_h, u_h)$ denote the solutions to \eqref{eqn:numerical_scheme}. 
        For a given numerical solution, we compute, at a specified time $T$, the $L^2$ error and the $L^2$ error of the gradient seperately.  
        Here, we assume that $T$ coincides with one of the discrete times (i.e. $\rho_h(T, x) = \rho^n(x)$ for some $n$).  We denote the errors for a given mesh $h,k$ by $E_{h,k}^0, E_{h,k}^1$, respectively. 
        \begin{align*}
            E_{h,k}^0(T) &= \|\rho(T, \cdot) - \rho_h(T, \cdot)\|_{L^2}^2 +  \|\wt(T, \cdot) - \wt_h(T, \cdot)\|_{L^2}^2  +  \|u(T, \cdot) - u_h(T, \cdot)\|_{L^2}^2, \\
            E_{h,k}^1(T) &= \|\partial_x \rho(T, \cdot) - \partial_x \rho_h(T, \cdot) \|_{L^2}^2 +  \|\partial_x \wt(T, \cdot) - \partial_x \wt_h(T, \cdot) \|_{L^2}^2  \\
                &+  \| \partial_x u(T, \cdot) - \partial_x u_h(T, \cdot) \|_{L^2}^2
        \end{align*}
        To illustrate convergence of the numerical scheme in $H^1$, we perform the following test. 
        Fix $T = 0.5$.  Vary the spatial and temporal mesh size simultaneously, $h_i = \frac{1}{2^i}$ with $2 \leq i \leq 7$, $k_i = h_i/4$. 
        For each $(h,k)$, the errors $E_{h,k}^0(T)$ and $E_{h,k}^1(T)$ are computed and the loglog graph of the errors with respect to $h$ (with the understanding that $k = h/4$) is shown in table \ref{plots:vary_k_and_h}. 

        \begin{table}[!h]
            \begin{center}
                \begin{tabular}{cc}
                    \includegraphics[width=0.47\linewidth]{\detokenize{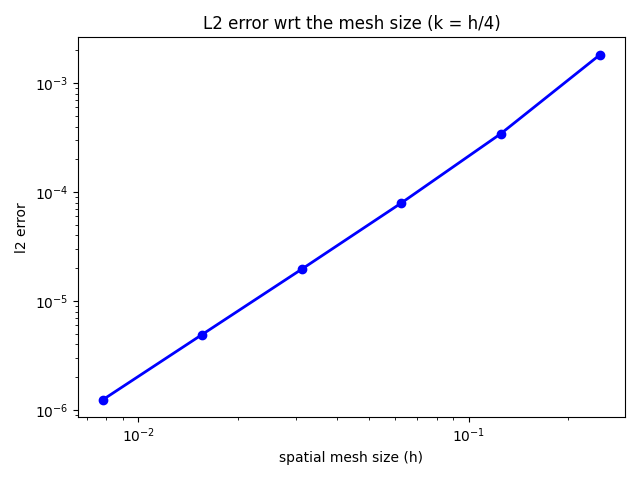}} &
                    \includegraphics[width=0.47\linewidth]{\detokenize{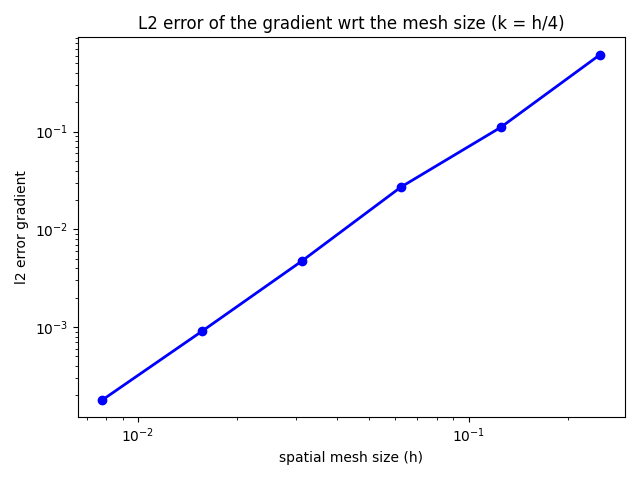}} 
                \end{tabular}
            \end{center} 
            \caption{Loglog plot of the $L^2$ error and $L^2$ error of the gradient with respect to $h$, where $k = h/4$, i.e. for fixed $T = 0.5$,  
            the mesh sizes are $(h,k) = (\frac{1}{4}, \frac{1}{16}), (\frac{1}{8}, \frac{1}{32}), (\frac{1}{16}, \frac{1}{64}), (\frac{1}{32}, \frac{1}{128}), (\frac{1}{64}, \frac{1}{256})$.}
            \label{plots:vary_k_and_h}
        \end{table}

        \begin{remark}
            \label{rmk:courant_condition}
            In the numerical experiments for \eqref{eqn:numerical_scheme_with_forcing}, a Courant-Friedrichs-Lewy condition for the mesh parameters $h,k$ was observed. 
            In some cases, when $k \approx h$, numerical instability was observed.  
            We found experimentally that small values of $k$ when compared to $h$ suppressed these instabilities and, for this particular experiment, $k = h/4$ sufficed. 
            However, the stability region is not known and it is possible that the stability region is more strict for smaller mesh sizes (for instance, $k \leq C h^2$). 
        \end{remark}

        \begin{remark}
            According to Table \ref{plots:vary_k_and_h}, it appears that the convergence rate of the $L^2$ error and the $L^2$ error of the gradient are similar (while one 
            would expect the convergence rate of the gradient to be slower).  However, the mesh sizes used here may not be fine enough to observe a clear convergence rate.  
            The purpose of this experiment is to illustrate that the error is approaching zero, regardless of the rate. 
        \end{remark}

\section{Conclusion and Future Work}
We have extended many important classical results about the Cucker-Smale model to the more versatile $\st$-model with adaptive strength and Favre averaging protocol \eqref{EAS_WM}. 
In order to gain versatility of behavior, it sacrificed the conservation of momentum and the energy law. 
Nonetheless, we showed that it still retains many of the desirable analytical qualities of the Cucker-Smale model-- namely
alignment, local well-posedness, a threshold condition 
for global well-posedness in 1D, existence for small data and uni-directional flocks, strong flocking, and relative entropy estimates on the limiting flock.

Although such results were obtained for the $\st$-model with Favre averaging, conceivably an extension to the general environmental averaging model in the sense of \cite{Sbook} is possible.  Indeed, with the appropriate assumptions on the 
averaging $[\u]_{\rho}$, one expects to obtain alignment results, local well-posedness, and a threshold condition 
for global well-posedness for uni-directional flocks.  However, the small data, strong flocking, and entropy estimates 
depends on the exponential decay of the derivatives of the Favre averaged velocity $\u_F$, which incidentally, depends on the algebraic structure 
of the Favre averaging.  Such results will therefore not be able to be easily extended to more general averaging 
operators $[\u]_{\rho}$ in the $\st$-model \eqref{s_model}.

While not as general as \eqref{s_model}, \eqref{EAS_WM} provides an important unification of the Cucker-Smale and Motsch-Tadmor models 
into one which retains many of the desirable analytic and qualitative features of both.
Unlike Cucker-Smale and Motsch-Tadmor, it was conceived at the hydrodynamic level and lacks a discrete and kinetic description. 
It may therefore be of general interest to put \eqref{EAS_WM} on more firm physical and theoretical grounds by researching its  
discrete and kinetic counterparts from which it arises.

\appendix
\section{Appendix}
\label{appendix}
The invariance and contractivity estimates on the map $\mathcal{F}$ for the local 
well-posedness argument in Section \ref{LWP} 
use the analyticity property of the heat semigroup and non-linear estimates on the derivatives, which we will 
record here.
For the non-linear estimates on derivatives, the Faa di Bruno Formula and the Gagliardo-Nirenburg inequality are used 
to obtain estimates on $\|[\u_1]_{\rho_1} - [\u_1]_{\rho_2} \|_{H^l}$ and $\| \u_F \|_{H^l}$, 
where $\u_F = (\u \rho)_{\phi} / \rho_{\phi}$ is the Favre-Filtration.  
Due to the presence of $\rho_{\phi}$ in the denominator of the Favre-filtration, 
an estimate on the Sobolev norm of the reciprocal is necessary.  
The estimate \eqref{recip_rho_estimate} on the Sobolev norm of $1/f$ 
is a specific case of Lemma 2.5 of \cite{lear_reynolds_shvydkoy2020local}, which 
estimates the fractional Sobolev norm for the purpose of showing local well-posedness of topological models. 
However, we record a simplified version here in order 
to highlight the dependence of the estimates on $\|1/\rho\|_{\infty}$ and to 
avoid the unnecessary dependence on $\nabla \rho$.  
We also record the Csisz\'{a}r-Kullback inequality used in the entropy estimates 
in Section \ref{sec:entropy}. 


\subsection{Classical Lemmas} 
We record the Gagliardo-Nirenberg inequality, Csisz\'{a}r-Kullback inequality, and analyticity property of 
the heat semi-group.

\begin{lemma} [Gagliardo-Nirenburg Inequality]
    \label{lma:GN_inequality}
    Assume the domain is $\T^n$ or $\R^n$. If $1 \leq q \leq \infty$, $0 \neq j < m$ integers, $1 \leq r \leq \infty$, $p \geq 1$ and 
    $0 \leq \theta < 1$ such that 
    \begin{equation*}
        \frac{1}{p} = \frac{j}{n} + \theta \Big(\frac{1}{r} - \frac{m}{n} \Big) + \frac{1-\theta}{q}, \hspace{5mm} \frac{j}{m} \leq \theta < 1
    \end{equation*}
    then there's a constant $C := C(j,m,n,q,r,\theta)$ such that 
    \begin{equation*}
        \|D^j f\|_{L^p} \leq C \|D^m f\|_{L^r}^{\theta} \|f\|_{L^{q}}^{1-\theta} 
    \end{equation*}
    for any $f \in L^q \cap L^2 \cap W^{m,r}$.
    For our purposes, we set $r = 2$ and place the smallest power on the $W^{m,r}(\R^n)$ norm, i.e. $\theta = j/m$. 
    We obtain 
    \begin{equation*}
        \|f\|_{W^{j,p}} \leq C \|f\|_{H^m}^{j/m} \|f\|_{L^q}^{1-j/m}
    \end{equation*}
    where 
    \begin{equation*}
        \frac{1}{p} = \frac{j}{2m} + \frac{1}{q} \Big(1 - \frac{j}{m} \Big)
    \end{equation*}   
\end{lemma}

\begin{lemma} [Csisz\'{a}r-Kullback Inequality]
    \label{lma:ck_inequality}
    The entropy is given by $\cH = \int_{\T} \rho \log \frac{\rho}{\bar{\rho}} \dx$ where $\bar{\rho} = \frac{1}{2\pi} \int_{\T} \rho(x) \dx$. 
    Then
    \begin{equation*}
        \frac{1}{4\pi} \|\rho - \bar{\rho}\|_{L^1}^2 \leq \bar{\rho} \cH \leq \|\rho - \bar{\rho}\|_{L^2}^2
    \end{equation*}
\end{lemma}

\begin{lemma} [Analyticity Property of Heat Semigroup]
    For all $\epsilon, t > 0$, there exists a constant $C > 0$ independent of $f$, $\epsilon$, $t$ such that 
    \begin{equation*} 
        \|\partial e^{\epsilon t \Delta} f \|_p \leq \frac{C}{\sqrt{\epsilon t}} \|f\|_p  \hspace{10mm} 1 \leq p \leq \infty
    \end{equation*}
\end{lemma}

\subsection{Non-linear estimates on derivatives}
\label{appdx:non-linear_estimates}
To establish the aforementioned non-linear estimates on derivatives, we first record the Fa di Bruno Formula.
For a more general version of \ref{recip_rho_estimate} on a fractional Sobolev space, we refer to Lemma 2.5 of 
\cite{lear_reynolds_shvydkoy2020local}.

\spnewtheorem*{formula}{Faa Di Bruno's Formula}{\bfseries}{\itshape}
\begin{formula}
    Let $f^{(i)}$ denote the $i^{th}$ partial derivative of $f$. 
    \begin{equation*}
        \partial^P h(g) = \sum_{\textbf{j}} \frac{P!}{j_1!1!^{j_1} j_2!2!^{j_2} \dots j_P!P!^{j_P}} h^{(j_1 + \dots + j_p)}(g) 
                                \prod_{i=1}^{|\textbf{j}|} g^{(k_i)}
    \end{equation*}
    where the sum is over all $P$-tuples of non-negative integers $\textbf{j} = (j_1, \dots, j_P)$ satisfying 
    \begin{equation*}
        1*j_1 + 2*j_2 + 3*j_3 + \dots + P*j_P = P
    \end{equation*}
    and 
    \begin{equation*}
        k_1 + k_2 + \dots + k_{|\textbf{j}|} = P 
    \end{equation*}
\end{formula}

\begin{proposition} [$H^l$ estimate on $1/f$]
\label{recip_rho_estimate}
    Assume the domain is $\T^n$ or $\R^n$.  For $f \in H^l$, there's a constant $C := C(l,n)$ such that 
    \begin{equation*}
        \Big\| \dl \Big( \frac{1}{f} \Big) \Big\|_2 
            \leq C \Big\| \frac{1}{f} \Big\|_{\infty}^{l+1} \| f \|_{H^l}
    \end{equation*}
    and
    \begin{equation*}
        \Big\| \dl \Big( \frac{\nabla f}{f} \Big) \Big\|_2 
            \leq C \Big( \|f\|_{H^{l+1}} \Big\|\frac{1}{f} \Big\|_{\infty} 
                    + \|\nabla f\|_{\infty} \Big\| \frac{1}{f} \Big\|_{\infty}^{l+1} \| f \|_{H^l} \Big) 
    \end{equation*}
    \begin{proof}
        Using $h(x) = \frac{1}{x}$ and $g(x) = f(x)$ in Faa di Bruno's formula and Holder's Inequality, we have 
        for some constant $C' := C'(l)$, which may change from line to line, 
        \begin{align*}
            \Big\| \dl \Big( \frac{1}{f} \Big) \Big\|_2
                &= \Big\| \sum_{\textbf{j}} \frac{l!}{j_1!1!^{j_1} j_2!2!^{j_2} \dots j_l!l!^{j_l}} \frac{(-1)^{j_1 + \dots + j_l}}{f^{j_1 + \dots + j_l + 1}} 
                            \prod_{i=1}^{|\textbf{j}|} f^{(k_i)} \Big\|_2 \\
                &\leq C' \Big\| \frac{1}{f} \Big\|_{\infty}^{l+1}  \prod_{i=1}^{|\textbf{j}|} \| (\partial^{k_i} f) \|_{L^{p_i}} \\
                &\leq C' \Big\| \frac{1}{f} \Big\|_{\infty}^{l+1}  \prod_{i=1}^{|\textbf{j}|} \| f \|_{W^{k_i, p_i}} \\
        \end{align*}
        where $\sum_{i=1}^{|\textbf{j}|} k_i = l$ and $\sum_{i=1}^{|\textbf{j}|} 1/p_i = 1/2$. 
        Choosing $\frac{1}{p_i} = \frac{k_i}{2l}$ and $q=\infty$ in the Gagliardo-Nirenberg 
        inequality \ref{lma:GN_inequality}, we obtain for some constant $C := C(l, n)$, which may change from line to line 
        \begin{equation*}
            \leq C \Big\| \frac{1}{f} \Big\|_{\infty}^{l+1} \prod_{i=1}^{|\textbf{j}|} \| f \|_{H^l}^{k_i/l} 
            = C \Big\| \frac{1}{f} \Big\|_{\infty}^{l+1} \| f \|_{H^l}\\            
        \end{equation*}
        Using this and the product estimate, we can estimate for $f \in H^{l+1}$,
        \begin{align*}
            \Big\| \dl \Big( \frac{\nabla f}{f} \Big) \Big\|_2 
                &\leq C \Big( \|f\|_{H^{l+1}} \Big\|\frac{1}{f} \Big\|_{\infty} + \|\nabla f\|_{\infty} \Big\|\frac{1}{f} \Big\|_{H^l} \Big) \\
                &\leq C \Big( \|f\|_{H^{l+1}} \Big\|\frac{1}{f} \Big\|_{\infty} + \|\nabla f\|_{\infty} \Big\| \frac{1}{f} \Big\|_{\infty}^{l+1} \| f \|_{H^l} \Big) \\
        \end{align*}
    \end{proof}    
\end{proposition}

\begin{proposition} [$H^l$ Contractivity estimate on Favre Filtration]
\label{contractivity_estimate}
    Let $\u_F = (\u \rho)_{\phi} / \rho_{\phi}$. Then there exists a constant $C := C(l,n)$ such that 
    \begin{align*}
        \| [\u_1]_{\rho_1} - [\u_1]_{\rho_2} \|_{H^{l}} 
            &\leq \Big\| \frac{1}{{\rho_1}_{\phi} {\rho_2}_{\phi}} \Big\|_{\infty} 
                    \| (\u_1 \rho_1)_{\phi} ({\rho_2})_{\phi} - (\u_1 \rho_2)_{\phi} ({\rho_1})_{\phi} \|_{H^{l}} \\
                &+ C \Big\| \frac{1}{{\rho_1}_{\phi} {\rho_2}_{\phi}} \Big\|_{\infty}^{l+1} \| {\rho_1}_{\phi} {\rho_2}_{\phi} \|_{H^{l}}
                \| (\u_1 \rho_1)_{\phi} ({\rho_2})_{\phi} - (\u_1 \rho_2)_{\phi} ({\rho_1})_{\phi} \|_{\infty}
    \end{align*} 
    \begin{proof}
        Using the commutator estimate, we obtain
        \begin{align*}
        \| [\u_1]_{\rho_1} - [\u_1]_{\rho_2} \|_{H^l} 
            &= \Big\| \frac{ (\u_1 \rho_1)_{\phi} ({\rho_2})_{\phi} 
                - (\u_1 \rho_2)_{\phi} ({\rho_1})_{\phi}}    { {\rho_1}_{\phi} {\rho_2}_{\phi} } \Big\|_{H^l} \\ 
            &\leq \Big\| \frac{1}{{\rho_1}_{\phi} {\rho_2}_{\phi}} \Big\|_{\infty} 
                \| (\u_1 \rho_1)_{\phi} ({\rho_2})_{\phi} - (\u_1 \rho_2)_{\phi} ({\rho_1})_{\phi} \|_{H^l} \\
            &+ \Big\| \frac{1}{{\rho_1}_{\phi} {\rho_2}_{\phi}} \Big\|_{H^l} 
                \| (\u_1 \rho_1)_{\phi} ({\rho_2})_{\phi} - (\u_1 \rho_2)_{\phi} ({\rho_1})_{\phi} \|_{\infty} \\
        \end{align*}
        Then \ref{recip_rho_estimate} applied to $1/({\rho_1}_{\phi} {\rho_2}_{\phi})$ gives 
        \begin{equation*}
            \Big\| \frac{1}{{\rho_1}_{\phi} {\rho_2}_{\phi}} \Big\|_{H^l} 
                \leq C \Big\| \frac{1}{{\rho_1}_{\phi} {\rho_2}_{\phi}} \Big\|_{\infty}^{l+1} \| {\rho_1}_{\phi} {\rho_2}_{\phi} \|_{H^l}
        \end{equation*}
        and the desired inequality follows. 
    \end{proof}
\end{proposition}

\begin{proposition} [$H^l$ estimate on Favre Filtration]
    \label{favre_estimate}
    Let $\u_F = (\u \rho)_{\phi} / \rho_{\phi}$. Then there's a constant $C := C(l,m)$
    \begin{equation*}
        \| \u_F \|_{H^l} 
        \leq \Big\| \frac{1}{\rho_{\phi}} \Big\|_{\infty} \|(\u \rho)_{\phi} \|_{H^l}
        + C \Big\| \frac{1}{\rho_{\phi}} \Big\|_{\infty}^{l+1} \| \rho_{\phi} \|_{H^l} \|(\u \rho)_{\phi} \|_{\infty} 
    \end{equation*}

    \begin{proof}
        Apply the commutator estimate and the Gagliardo-Nirenberg inequality as in \ref{contractivity_estimate}.
    \end{proof}
\end{proposition}

\clearpage

\end{document}